\documentclass[11pt,reqno]{amsart}
\usepackage{amssymb,amsmath,amsthm}
\usepackage{graphicx}
\usepackage{subfigure}
\usepackage{color}
\newtheorem{theorem}{Theorem}[section]
\newtheorem{lemma}[theorem]{Lemma}
\newtheorem{proposition}[theorem]{Proposition}

\theoremstyle{definition}

\theoremstyle{remark}
\newtheorem{remark}[theorem]{Remark}

\makeatletter
\@namedef{subjclassname@2020}{2020 Mathematics Subject Classification}
\makeatother

\numberwithin{equation}{section}

\subjclass[2020]{Primary  34D20, 35B10; Secondary  34C23, 34C25. }

\keywords{Patch model,  destabilization, periodic solution, bifurcation. }
\date{\today}
\thanks{ \dag Email: schen@ccnu.edu.cn.}
\thanks{\ddag Corresponding author: hjc@mail.ccnu.edu.cn.}

\begin{document}

\title[Destabilization of synchronous periodic solutions]
{Destabilization of synchronous periodic solutions for patch models:
a criterion by period functions}
\thanks{This work was partly supported by the National Natural Science Foundation of China (Grant No. 12101253, 12231008)
and the Scientific Research Foundation of CCNU (Grant No. 31101222044).}

\maketitle

\vskip 1pt
\centerline{\scshape Shuang Chen$^{\,\dag}$, Jicai Huang$^{\,\ddag}$}
\medskip
{\footnotesize
\centerline{School of Mathematics and Statistics, and Hubei Key Laboratory of Mathematical Sciences,}
\centerline{Central China Normal University, Wuhan, Hubei 430079, China}
} 

\medskip

\begin{abstract}
In this paper, we study the destabilization of synchronous periodic solutions
for patch models. By applying perturbation theory for matrices,
we derive asymptotic expressions of the Floquet spectra
and provide a destabilization criterion for synchronous periodic solutions arising from
closed orbits or degenerate Hopf bifurcations in terms of period functions.
Finally, we apply the main results to the well-known two-patch Holling-Tanner model.
\end{abstract}

\section{Introduction}

Patch models have been extensively used to understand the spatial spread of infectious diseases and the effect of population dispersal on
the total abundance and the total populations distribution (see, for instance, \cite{Allen-07,Arumugam2020,Gao-12,Gao-19,Gao-20,Gao-Lou-21,Lengyel-Epstein-91,van-18,Wang-Zhao-04}
and the references therein).

In this paper,
we investigate a general $n$-patch model with cross-diffusion-like couplings:

\begin{eqnarray}\label{gpath}
\ \
\begin{aligned}
\frac{du_{1}^{(j)}}{dt} &=
    \delta \sum_{l=1}^{m}\sum_{i=1}^{n} \left(d_{1l}\left(u_{l}^{(i)}-u_{l}^{(j)}\right)\right)
     +f_{1}(u_{1}^{(j)},u_{2}^{(j)},\cdot\cdot\cdot, u_{m}^{(j)}),\ \ j=1,2,...,n,\\
\frac{du_{2}^{(j)}}{dt}  &=
    \delta \sum_{l=1}^{m}\sum_{i=1}^{n} \left(d_{2l}\left(u_{l}^{(i)}-u_{l}^{(j)}\right)\right)
     +f_{2}(u_{1}^{(j)},u_{2}^{(j)},\cdot\cdot\cdot, u_{m}^{(j)}),\ \ j=1,2,...,n,\\
&\ \  \  \cdots \cdots   \\
\frac{du_{m}^{(j)}}{dt}  &=
    \delta \sum_{l=1}^{m}\sum_{i=1}^{n} \left(d_{ml}\left(u_{l}^{(i)}-u_{l}^{(j)}\right)\right)
     +f_{m}(u_{1}^{(j)},u_{2}^{(j)},\cdot\cdot\cdot, u_{m}^{(j)}),\ \ j=1,2,...,n.
\end{aligned}
\end{eqnarray}
In the setting of  population dynamics,
the state variables $u_{i}^{(j)}$ are the  population densities of $i$-th species in the $j$-th patch,
$n\geq 2$ denotes the number of the patches,
$d_{ij}$ are  the diffusion coefficients,
and $\delta>0$ indicates the coupling strength.
Let
\begin{equation*}
\mathcal{U}(t)=
\left(
\begin{array}{c}
U^{1}\\
U^{2}\\
\vdots \\
U^{n}
\end{array}
\right), \
\mathcal{F}(\mathcal{U}(t))=
\left(
\begin{array}{c}
F(U^{1})\\
F(U^{2})\\
\vdots \\
F(U^{n})
\end{array}
\right), \
\mathcal{E}=\left(
\begin{array}{cccc}
(n-1)E & -E & \cdots & -E\\
-E & (n-1)E & \cdots & -E\\
\vdots & \vdots &\ddots & \vdots \\
-E  & -E  & \cdots & (n-1)E
\end{array}
\right),
\end{equation*}
where $U^{j}=(u_{1}^{(j)},u_{2}^{(j)},\cdot\cdot\cdot, u_{m}^{(j)})^{T}$, $F=(f_1,f_2,\cdot\cdot\cdot, f_m)^{T}$
and $E=(d_{kl})_{m\times m}$ is a $m\times m$ matrix.
Then we can rewrite patch model \eqref{gpath} in the compact form
\begin{eqnarray}\label{eq-patch}
\frac{d}{dt}\mathcal{U}(t)=-\delta \mathcal{E} \mathcal{U}(t)+\mathcal{F}(\mathcal{U}(t)).
\end{eqnarray}
Throughout this paper, we assume that  $f_{i}$ ($i=1,...,m$) are sufficiently smooth,
and use $^{T}$ to denote the transpose of a matrix or a vector.

A solution $\mathcal{U}_{0}$ of patch model \eqref{eq-patch} is called a synchronous periodic solution
if $\mathcal{U}_{0}=(\phi^{T},...,\phi^{T})^{T}$ and
$\phi:\mathbb{R}\to \mathbb{R}^{m}$ is a periodic function with the minimum period $P>0$.
In this case, this periodic function $\phi$ is also a periodic solution of the underlying kinetic system
\begin{eqnarray}\label{ODE-alpha}
\frac{d U}{dt}=:\dot U=F(U),
\end{eqnarray}
where $U=(u_{1},u_{2},...,u_{m})^{T}\in\mathbb{R}^{m}$ and $F(U)=(f_1(U),f_2(U),\cdot\cdot\cdot, f_m(U))^{T}$.
Furthermore, if $\phi$ is a (Lyapunov) stable periodic solution of the kinetic system \eqref{ODE-alpha},
then the corresponding synchronous periodic solution $\mathcal{U}_{0}$ is also stable
in system \eqref{eq-patch} without the cross-diffusion-like couplings.
A natural question arises:
\begin{enumerate}
\item[$\bullet$] Can the synchronous periodic solution $\mathcal{U}_{0}(t)$ become unstable in system \eqref{eq-patch} with the cross-diffusion-like couplings?
\end{enumerate}
The instability driven by the cross-diffusion-like couplings  is called
the destabilization of synchronous periodic solutions for patch model \eqref{gpath}.
It is also called the Turing instability of periodic solutions \cite{Yi-21},
in order to celebrate  Turing's discovery in \cite{Turing-52}, i.e.,
diffusion could destabilize stable equilibrium solutions of reaction-diffusion systems.

We are interested in the destabilization of synchronous periodic solutions for patch model \eqref{gpath}.
This is directly motivated by various phenomena and problems arising from real-world applications.
For example,
\cite{Arumugam-2021} recently shown that unstable states play a vital role in transient dynamics and
the resilience of ecological systems to environmental change.
\cite{Lengyel-Epstein-91} once found that
the destabilization of periodic solutions in chemically reacting systems
can lead to complicated oscillations and chaos.
It is significant to understand the effect of the connectivity of subregions on infectious disease transmission \cite{Gao-12,Gao-Lou-21}.
Along this direction,
we also need to further investigate the impact of the cross-diffusion-like couplings on periodic oscillations.

The main obstacle to investigate the destabilization of synchronous periodic solutions is that
it is difficult to analyze the Floquet spectra of the related linearizations about synchronous periodic solutions.
The obstacle becomes evident after we present the linearization of patch model \eqref{gpath}
about a synchronous periodic solution $\mathcal{U}_{0}$, i.e., the following periodic system
\begin{eqnarray}\label{eq-patch-var}
\frac{d}{dt}Y(t)=(-\delta \mathcal{E} +\widehat{J}(t))Y(t),
\end{eqnarray}
where $Y(t)\in\mathbb{R}^{mn}$, $\widehat{J}(t)={\rm diag}(J(t),J(t),\cdot\cdot\cdot,J(t))$ and $J(t)=F_{U}(\phi(t))$.
Note that general patch models always involve multiple patches.
Then the related periodic system \eqref{eq-patch-var} is high-dimensional, and it is challenging
to give the explicit expressions of the Floquet spectra
for high-dimensional periodic systems, even for three-dimensional systems.
See some reviews in such as \cite{Chen-Duan-22,Chen-Huang-23,Hale-80,Yi-21}.

Bifurcations of invariant sets (e.g. equilibria, periodic solutions, homoclinic loops, heteroclinic loops, etc)
in a parametrically perturbed system give rise to periodic solutions.
The period of bifurcating periodic solutions arising from an invariant set
can be well defined in terms of system parameters when parameters are near a bifurcation point.
This gives the period function of bifurcating periodic solutions \cite{Cod-Lev-55,Gasull-05}.
In our recent work \cite{Chen-Huang-23},
we provided a criterion for the destabilization of synchronous periodic solutions bifurcating from double homoclinic loops.
Based on the Lyapunov-Schmidt reduction, we obtained the characteristic function to determine the Floquet spectra associated with synchronous periodic solutions,
while the period functions of bifurcating periodic solutions are not well-defined at bifurcation points \cite{Chen-Huang-23,Gasull-05}.
It is also interesting and challenging to deal with bifurcating periodic solutions
whose period function is at least continuously differentiable.
Some typical examples include periodic solutions arising from the parametric perturbations of equilibria and limit cycles.
See \cite{Gasull-05} for instance.

Our goal is to further give criteria for the destabilization of synchronous periodic solutions
in the term of period functions if the related period functions are continuously differentiable.
Similar criteria were previously proposed
for the diffusion-derived instability of spatially homogeneous periodic solutions in  reaction-diffusion systems.
For example,
Maginu \cite{Maginu-79} in 1979 and Ruan \cite{Ruan-98a} in 1998
once considered the diffusion-derived instability of spatially homogeneous periodic solutions for reaction-diffusion systems
in the entire space.
They established the relation between the Floquet spectra and the period functions of bifurcating periodic solutions for
some certain perturbations of the kinetic systems.
After that, they gave criteria for
the instability of spatially homogeneous periodic solutions  in terms of  the dominant term of the related period functions
(see Lemma 2 in \cite{Maginu-79} and Formula (4.16) in \cite{Ruan-98a}).

In order to give a criterion for the destabilization,
we consider a single parametric perturbation of the  kinetic system \eqref{ODE-alpha} as follows:
\begin{eqnarray}\label{pert-ODE}
(I_{m}+\epsilon E)\frac{d U}{dt}=F(U),
\end{eqnarray}
where $\epsilon\in\mathbb{R}$ is a small parameter, and $E$ and $F$ are defined as in \eqref{eq-patch}.
If the kinetic system \eqref{ODE-alpha} has a hyperbolic periodic solution $\phi$,
then by the classical bifurcation theory \cite{Andronov-etal-73,Cod-Lev-55,Ye-86},
there exists a sufficiently small $\epsilon_{0}>0$ such that
the perturbed system \eqref{pert-ODE} admits a family of bifurcating periodic solutions $\phi_{\epsilon}$ for $\epsilon\in (-\epsilon_{0},\epsilon_{0})$
 that bifurcate from the hyperbolic periodic solution $\phi$.
Let $P(\epsilon)$ denote the minimum period of $\phi_{\epsilon}$, and call
$P(\cdot)$ the period function of the family of bifurcating periodic solutions $\phi_{\epsilon}$.
Furthermore, the period function $P(\epsilon)$ is continuously differentiable in the open set $(-\epsilon_{0},\epsilon_{0})$.
By applying the dominant term of the period function $P(\epsilon)$,
we give a criterion for the destabilization of synchronous periodic solutions (see Theorem \ref{lm-unstab-patch}).
The argument is mainly based on  the perturbation theory for matrices that was developed in our recent work \cite{Chen-Huang-23}.
Actually, we present the asymptotic expression of the related Floquet spectra for synchronous periodic solutions.

It is worth mentioning that our result  actually improves  Yi's work \cite{Yi-21}.
Yi recently introduced  a single parametric perturbation of patch model \eqref{gpath} as follows:
\begin{eqnarray}\label{eq:pert-patch}
(I_{mn}+\epsilon \mathcal{E})\frac{d}{dt}\mathcal{U}(t)=\mathcal{F}(\mathcal{U}(t)),
\end{eqnarray}
where $I_{mn}$ is the $mn\times mn$ identity matrix,
and $\mathcal{E}$ and $\mathcal{F}$ are defined as in \eqref{eq-patch}.
Under the assumption that
the perturbed patch model \eqref{eq:pert-patch}
possesses a family of bifurcating periodic solutions
$$((U_{p}^{1}(\cdot,\epsilon))^{T},(U_{p}^{2}(\cdot,\epsilon))^{T},...,(U_{p}^{n}(\cdot,\epsilon))^{T})^{T}$$
with period functions $\widehat{P}(\epsilon)$  that vanish asymptotically in one patch and persist in the other $(n-1)$ patches, i.e.,
\begin{eqnarray*}
((U_{p}^{1}(t,\epsilon))^{T},(U_{p}^{2}(t,\epsilon))^{T},...,(U_{p}^{n}(t,\epsilon))^{T})^{T}
\to ((\phi(t))^{T},(\phi(t))^{T},\cdot\cdot\cdot,(\phi(t))^{T},0)^{T}
 \ \ \ \mbox{ as }\ \epsilon\to 0,
\end{eqnarray*}
uniformly in $t\in\mathbb{R}$.
To determine the destabilization of synchronous periodic solutions for patch model \eqref{gpath},
following the idea of Maginu \cite{Maginu-79} and Ruan \cite{Ruan-98a},
Yi \cite{Yi-21} presented a criterion for the destabilization of
the synchronous periodic solution $\mathcal{U}_{0}=(\phi^{T},...,\phi^{T})^{T}$  in terms of
the first-order derivative of $\widehat{P}(\epsilon)$ at $\epsilon=0$.
Here, we give a simpler criterion  that only requires $P'(0)<0$,
where $P'(0)$ denotes the first-order derivative of $P(\epsilon)$ at $\epsilon=0$,
without any additional conditions. See Theorem 2.5.

The present paper  is devoted to the destabilization of synchronous periodic solutions for  patch models
and appears to be our second paper on this topic.
We refer  to our first paper \cite{Chen-Huang-23},
where the related period functions are not continuously differentiable at bifurcations points.
 Instead, here we explore the case of continuously differentiable period functions.
 The theory we developed is applied to give criteria for
the destabilization of synchronous periodic solutions,
bifurcating from closed orbits \cite{Urabe-54,Ye-86} and degenerate Hopf bifurcation \cite{Chow-etal-94,Far-etal-89}, in $n$-patch models
with two-dimensional kinetic systems.
It was once found that diffusion-driven instability of periodic solutions for reaction-diffusion systems
can not be induced by the identical diffusion rates \cite{Henry-81}.
As an easy by-product, we prove that patch model with the identical diffusion rates
never undergoes the destabilization of synchronous periodic solutions.

This paper is organized as follows.
The criterion for the destabilization of synchronous periodic solutions is given in section \ref{sec-mainresult},
and then we apply our results to general $n$-patch models with two-dimensional kinetic systems.
Sections \ref{sec:4.1} and \ref{sec-small}  are devoted to synchronous periodic solutions
arising from closed orbits and degenerate Hopf bifurcation, respectively.
Finally,  the well-known two-patch Holling-Tanner model is provided in Section \ref{sec:4.3} to  illustrate the main results.

\section{Destabilization of synchronous periodic solutions}
\label{sec-mainresult}

In this section,
we give the main result on the destabilization of synchronous periodic solutions.
The proof is based on two fundamental lemmas on the perturbation theory for matrices
 that were developed in our recent work \cite{Chen-Huang-23}.
We present them in Appendix A for convenience.

Let $\phi$ be a periodic solution with  the minimum period $P>0$ for the kinetic system \eqref{ODE-alpha}.
Then the linearization of the kinetic system \eqref{ODE-alpha} about this periodic solution $\phi$ is governed by
\begin{eqnarray}\label{ODE-var}
\frac{dU}{dt}=F_{U}(\phi(t)) U= J(t)U, \ \ \ \ \ \ U\in\mathbb{R}^{m}.
\end{eqnarray}
Throughout this section, we make the following assumption:
\begin{enumerate}
\item[{\bf (H)}] All Floquet multipliers $\gamma_{1},...,\gamma_{m}$ of the linearized system \eqref{ODE-var} satisfy
\begin{eqnarray}\label{cond-stab}
\gamma_{1}=1,\ \  |\gamma_{2}|< 1,..., |\gamma_{m}|< 1.
\end{eqnarray}
\end{enumerate}
Under this assumption, the periodic solution $\phi$ is stable with respect to the kinetic system \eqref{ODE-alpha}
and the synchronous periodic solution $\mathcal{U}_{0}(t)=(\phi(t)^{T},...,\phi(t)^{T})^{T}$ is also stable
with respect to patch model \eqref{eq:pert-patch} if $\delta=0$.
By the discussion in  the Introduction,
there exists a sufficiently small $\epsilon_{0}>0$ such that
the perturbed system \eqref{pert-ODE} with $\epsilon\in (-\epsilon_{0},\epsilon_{0})$
admits  a family of bifurcating periodic solutions $\phi_{\epsilon}$ for $\epsilon\in (-\epsilon_{0},\epsilon_{0})$
 that bifurcate from this stable solution $\phi$.
We use $P(\epsilon)$ to denote the related period function.

We are interested in the effect of the  cross-diffusion-like couplings on the stability of this synchronous periodic solution.
Consider the linearized system \eqref{eq-patch-var} of patch model \eqref{eq-patch} about the periodic solution $\mathcal{U}_{0}$.
 Note that  all Floquet multipliers of the linearized system \eqref{ODE-var}  satisfy \eqref{cond-stab}.
Then $n$ Floquet multipliers of system \eqref{eq-patch-var} with $\delta=0$  are one,
and all other Floquet multipliers  have  modulii less than one.
Let $\Psi(t,\delta)$ denote the principal fundamental matrix solution of system \eqref{eq-patch-var}.
Then we can compute that the kernel $\ker(\Psi(P, 0)-I_{mn})$ is spanned by the following vectors in $\mathbb{R}^{mn}$:
\begin{equation*}
\xi_{1}=
\left(
\begin{array}{c}
\phi_{1}(0)\\
0\\
\vdots \\
0
\end{array}
\right),\ \
\xi_{2}=
\left(
\begin{array}{c}
0\\
\phi_{1}(0)\\
\vdots \\
0
\end{array}
\right), \cdots, \
\xi_{n}=
\left(
\begin{array}{c}
0\\
0\\
\vdots \\
\phi_{1}(0)
\end{array}
\right),
\end{equation*}
where  $0$ is the zero vector in $\mathbb{R}^{m}$, $\xi_{j}\in\mathbb{R}^{mn}$ and $\phi_{1}$ is given by
\begin{eqnarray*}
\phi_{1}(t):=\frac{d}{dt}\phi(t), \ \ \ \ \ t\in\mathbb{R}.
\end{eqnarray*}
Here we use the fact that $\phi_{1}$ is a solution of the linearized system \eqref{ODE-var}.
Next we give an important property on the monodromy operator $\Psi(P,\delta)$.

\begin{lemma} \label{lm-expansion}
Suppose that $\phi(\cdot,\epsilon)$ $(\epsilon\in (-\epsilon_{0},\epsilon_{0}))$
are periodic solutions bifurcating from $\phi(\cdot)$ in the perturbed system \eqref{pert-ODE}.
Define
\begin{eqnarray*}
\tilde{\xi}_{1}:=(\xi_{1}+\xi_{2}+\cdot\cdot\cdot+\xi_{n}),\ \ \ \ \
\tilde{\xi}_{j}(\delta):=\frac{1}{n}(\xi_{j-1}-\xi_{n})+\delta(\eta_{j-1}-\eta_{n}),\ \ \ \ \ j=2,...,n,
\end{eqnarray*}
where $\delta\geq 0$ and $\eta_{j}$ for $j=1,...,n$ are given by
\begin{equation*}
\eta_{1}=
\left(
\begin{array}{c}
\frac{\partial}{\partial\epsilon}\phi(0,0)\\
0\\
\vdots \\
0
\end{array}
\right),\ \
\eta_{2}=
\left(
\begin{array}{c}
0\\
\frac{\partial}{\partial\epsilon}\phi(0,0)\\
\vdots \\
0
\end{array}
\right), \cdots, \
\eta_{n}=
\left(
\begin{array}{c}
0\\
0\\
\vdots \\
\frac{\partial}{\partial\epsilon}\phi(0,0)
\end{array}
\right).
\end{equation*}
Then
\begin{eqnarray}\label{Expand-1}
\Psi(P,\delta)\tilde{\xi}_{1}=\tilde{\xi}_{1}
\end{eqnarray}
and for sufficiently small $\delta\geq 0$,
\begin{eqnarray}\label{Expand-2}
\Psi(P,\delta)\tilde{\xi}_{j}(\delta)
=(1-\delta nP'(0))\tilde{\xi}_{j}(\delta)+O(\delta^{2}),\ \ \ \ \ j=2,...,n.
\end{eqnarray}

\end{lemma}
\begin{proof}
Note that
$\widetilde{\mathcal{U}}(t):=
\left(\left(\phi_{1}(t)\right)^{T},\left(\phi_{1}(t)\right)^{T},... ,\left(\phi_{1}(t)\right)^{T}\right)^{T}$
satisfies \eqref{eq-patch-var}.
Then the monodromy operator  $\Psi(P,\delta)$ has the Floquet multiplier $\gamma_{1}(\delta)\equiv1$ for each $\delta\geq 0$
and satisfies \eqref{Expand-1}.

In order to prove \eqref{Expand-2},
we first consider the case $j=2$.
Let $\Psi_{1}(t,\delta)$ be the unique solution of the initial value problem:
\begin{eqnarray}\label{IVP-1}
\begin{aligned}
\frac{d}{dt}\Psi_{1}(t,\delta)=&\ (-\delta \mathcal{E} +\widehat{J}(t))\Psi_{1}(t,\delta),\\
\Psi_{1}(0,\delta)=&\  \tilde{\xi}_{2}(\delta).
\end{aligned}
\end{eqnarray}
Differentiating with respect to $\delta$ and then setting $\delta=0$,
we have
\begin{eqnarray} \label{eq-imp-1-1}
\begin{aligned}
\frac{\partial}{\partial t}\frac{\partial}{\partial \delta}\Psi_{1}(t,0)
=&\ -\mathcal{E}\Psi_{1}(t,0)+\widehat{J}(t)\frac{\partial}{\partial \delta}\Psi_{1}(t,0),\\
\frac{\partial}{\partial \delta}\Psi_{1}(0,0)
=&\  \eta_{1}-\eta_{n}.
\end{aligned}
\end{eqnarray}
Note that $\Psi_{1}(t,0)$ satisfies
\begin{eqnarray*}
\frac{d}{dt}\Psi_{1}(t,0)= \widehat{J}(t)\Psi_{1}(t,0),
\ \ \ \ \ \Psi_{1}(0,0)=\frac{1}{n}(\xi_{1}-\xi_{n}).
\end{eqnarray*}
Then we can compute
\[
\Psi_{1}(t,0)
=
\left(
\begin{array}{c}
\frac{1}{n}\phi_{1}(t)\\
0\\
\vdots \\
0\\
-\frac{1}{n}\phi_{1}(t)
\end{array}
\right),\ \ \ \ \
\mathcal{E}\Psi_{1}(t,0)
=
\left(
\begin{array}{c}
E\phi_{1}(t)\\
0\\
\vdots \\
0\\
-E\phi_{1}(t)
\end{array}
\right),
\]
where $0$ is the zero vector in $\mathbb{R}^{m}$.

Recall that $\phi(t,\epsilon)$ is a periodic solution bifurcating
from $\phi(t)$ in the perturbed system \eqref{pert-ODE}.
Then substituting  $U(t)=\phi(t,\epsilon)$ into \eqref{pert-ODE}
and  differentiating system \eqref{pert-ODE} with respect to $\epsilon$,
we have
\begin{eqnarray}\label{eq-imp-2-1}
\frac{\partial }{\partial t}\left(\frac{\partial }{\partial \epsilon}\phi(t,0)\right)
=-E\phi_{1}(t)+J(t)\left(\frac{\partial }{\partial \epsilon}\phi(t,0)\right).
\end{eqnarray}
By the above equation, we can check  that the initial value problem \eqref{eq-imp-1-1}
has the unique solution
\begin{eqnarray*}
\frac{\partial}{\partial \delta}\Psi_{1}(t,0)
=
\left(
\begin{array}{c}
\frac{\partial }{\partial\epsilon}\phi(t,0)\\
0\\
\vdots \\
0\\
-\frac{\partial }{\partial\epsilon}\phi(t,0)
\end{array}
\right).
\end{eqnarray*}
Note that  $\Psi_{1}(P,\delta)$ is analytic in the parameter $\delta$.
We have the following  expansion:
\begin{eqnarray*}
\Psi_{1}(P,\delta)
   =\Psi_{1}(P,0)+\delta \frac{\partial}{\partial \delta}\Psi_{1}(P,0)+O(\delta^{2}).
\end{eqnarray*}
This together with \eqref{IVP-1} and the fact that
$$\Psi(P,0)\frac{\xi_{1}-\xi_{n}}{n}=\frac{\xi_{1}-\xi_{n}}{n}$$
yields
\begin{eqnarray}\label{Psi-expand}
\Psi(P,\delta)\tilde{\xi}_{2}(\delta)=\frac{1}{n}(\xi_{1}-\xi_{n})
+\delta\left(
\begin{array}{c}
\frac{\partial }{\partial\epsilon}\phi(P,0)\\
0\\
\vdots \\
0\\
-\frac{\partial }{\partial\epsilon}\phi(P,0)
\end{array}
\right)+O(\delta^{2}).
\end{eqnarray}
Since $\phi(t+P(\epsilon),\epsilon)=\phi(t,\epsilon)$,
differentiating with respect to $\epsilon$, we have
\begin{eqnarray*}
\frac{\partial }{\partial \epsilon}\phi(P,0)=-P'(0)\phi_{1}(0)+\frac{\partial }{\partial \epsilon}\phi(0,0).
\end{eqnarray*}
 Substituting this into \eqref{Psi-expand} yields
\begin{eqnarray*}
\Psi(P,\delta)\tilde{\xi}_{2}(\delta)
=(1-\delta nP'(0))\tilde{\xi}_{2}(\delta)+O(\delta^{2}).
\end{eqnarray*}
Similarly, we can prove that \eqref{Expand-2} holds for $j=3,...,n$.
This finishes the proof.
\end{proof}

In order to study the stability of $\mathcal{U}_{0}(t)$ with respect to system \eqref{eq-patch},
we give the following lemma on the Floquet multipliers of \eqref{pert-ODE}.
\begin{lemma}\label{lm-PF-patch}
For sufficiently small $|\epsilon|\geq 0$,
let $P(\epsilon)$ be the period function of
periodic solutions $\phi(t,\epsilon)$ bifurcating from $\phi(t)$ in the perturbed system \eqref{pert-ODE}.
Then for sufficiently small $\delta>0$,
system \eqref{pert-ODE} has $n$ Floquet multipliers $\gamma_{j}(\delta)$ perturbed from one,
and $\gamma_{j}(\delta)$ are in the form
\begin{eqnarray*}
\gamma_{1}(\delta)\equiv 1, \ \ \ \
\gamma_{j}(\delta)=1-nP'(0)\delta+O(\delta^{2}), \ \ \ j=2,...,n.
\end{eqnarray*}
\end{lemma}
\begin{proof}
Note that the monodromy operator $\Psi(p,\delta)$ is analytic in $\delta$.
Then for sufficiently small $\delta>0$,
the monodromy operator $\Psi(p,\delta)$ has  $n$ eigenvalues $\gamma_{j}(\delta)$ ($j=1,2,...,n$)
arising from one. This finishes the proof for the first statement.

By the proof of Lemma \ref{lm-expansion}, we have $\gamma_{1}(\delta)\equiv 1$ for $\delta\geq 0$.
To prove the expressions for $\gamma_{j}(\delta)=1-nP'(0)\delta+O(\delta^{2})$, $j=2,...,n$,
we define $W_{+}(\delta)$ and $\Lambda_{+}(\delta)$ by
$W_{+}(\delta)= \left(\tilde{\xi}_{1},\tilde{\xi}_{2}(\delta),...,\tilde{\xi}_{n}(\delta)\right)$ and
\begin{eqnarray*}
\Lambda_{+}(\delta)=\left(
\begin{array}{cccc}
1 & 0 & \cdot\cdot\cdot & 0\\
0 & 1-\delta nP'(0) & \cdot\cdot\cdot & 0\\
\cdot\cdot\cdot & \cdot\cdot\cdot & \cdot\cdot\cdot & \cdot\cdot\cdot\\
0 & 0 & \cdot\cdot\cdot & 1-\delta nP'(0)
\end{array}
\right),
\end{eqnarray*}
where  $\tilde{\xi}_{1}$, $\tilde{\xi}_{2}(\delta)$,...,$\tilde{\xi}_{n}(\delta)$ are defined in Lemma \ref{lm-expansion}.
Then by \eqref{Expand-1} and \eqref{Expand-2},
\begin{eqnarray}\label{df-W+}
\Psi(P,\delta)W_{+}(\delta)=W_{+}(\delta)\Lambda_{+}(\delta)+O(\delta^{2})
\end{eqnarray}
for sufficiently small $\delta\geq 0$.

By \eqref{cond-stab}, Lemma \ref{lm-app-0}  and the definition of $W_{+}(\delta)$,
there exists a small $\delta_{0}>$ and continuous functions
$\tilde{\xi}_{j}(\delta)\in \mathbb{R}^{mn}$  for $j=n+1,...,mn$ and $|\delta|<\delta_{0}$
such that
$$W_{-}(\delta):=\left(\tilde{\xi}_{n+1}(\delta),...,\tilde{\xi}_{mn}(\delta)\right)$$
satisfies
\begin{eqnarray}\label{df-W-}
\Psi(P,\delta)W_{-}(\delta)=W_{-}(\delta)\Lambda_{-}(\delta),\ \ \
|{\rm det}(W_{+}(\delta),W_{-}(\delta))|>0, \ \ \ |\delta|<\delta_{0},
\end{eqnarray}
where $\Lambda_{-}(\delta)\in\mathbb{R}^{(m-1)n\times (m-1)n}$ is non-singular and continuous in $\delta$,
and all eigenvalue of  $\Lambda_{-}(\delta)$ is  bounded away from $\gamma_{j}(\delta)$ in the complex plane for $|\delta|<\delta_{0}$.
By \eqref{df-W+} and \eqref{df-W-}, we get
\begin{eqnarray*}
\Psi(P,\delta)(W_{+}(\delta),W_{-}(\delta))
=(W_{+}(\delta),W_{-}(\delta))
\left\{\left(
\begin{array}{cc}
\Lambda_{+}(\delta) & 0\\
0 & \Lambda_{-}(\delta)
\end{array}
\right)+W_{0}(\delta)\right\},
\end{eqnarray*}
where $W_{0}(\delta)$ is continuous in $\delta$ and $W_{0}(\delta)=O(|\delta|^{2})$ for sufficiently small $|\delta|$.
Therefore, the proof is finished by Lemma \ref{lm-app}.
\end{proof}

Now we state the main result in the following.
\begin{theorem}\label{lm-unstab-patch}
Suppose that $\phi$ is a periodic solution with  the minimum period $P>0$ for the kinetic system \eqref{ODE-alpha}
 that satisfies assumption {\bf (H)}.
Let $P(\epsilon)$ denote the period function of bifurcating  periodic solutions
arising from $\phi$ in the perturbed system \eqref{pert-ODE}.
Then for sufficiently small $\delta>0$,
the synchronous periodic solution $\mathcal{U}_{0}(t)=(\phi(t)^{T},...,\phi(t)^{T})^{T}$ is unstable with respect to patch model \eqref{gpath}
if $P'(0)<0$.
\end{theorem}
\begin{proof}
If $P'(0)<0$, then by Lemma \ref{lm-PF-patch},
the Floquet multipliers $\gamma_{j}(\delta)$ $(j=2,...,n)$ satisfy
\begin{eqnarray*}
\gamma'_{j}(0)=-nP'(0)>0.
\end{eqnarray*}
Recall that $\gamma_{j}(0)=1$ for $j=2,...,n$.
Then for sufficiently small $\delta>0$,
the linearized system \eqref{eq-patch-var} has at least $(n-1)$ Floquet multipliers
that have modulii greater than one.
This proves that $\mathcal{U}_{0}(t)$ is unstable. Thus, the proof is now complete.
\end{proof}

\begin{remark}
More recently, Yi \cite{Yi-21} considered patch model \eqref{gpath}
and studied the destabilization  of the synchronous periodic solution $\mathcal{U}_{0}$
 using the perturbed patch model \eqref{eq:pert-patch}. In order to prove the instability of $\mathcal{U}_{0}$,
\cite{Yi-21} required that the perturbed patch model \eqref{eq:pert-patch}
has a periodic solution
$((U_{p}^{1}(\cdot,\epsilon))^{T},(U_{p}^{2}(\cdot,\epsilon))^{T},...,(U_{p}^{n}(\cdot,\epsilon))^{T})^{T}$
 that asymptotically vanishes in one patch and persists in the other $(n-1)$ patches, i.e.,
\begin{eqnarray*}
((U_{p}^{1}(t,\epsilon))^{T},(U_{p}^{2}(t,\epsilon))^{T},...,(U_{p}^{n}(t,\epsilon))^{T})^{T}
\to ((\phi(t))^{T},(\phi(t))^{T},\cdot\cdot\cdot,(\phi(t))^{T},0)^{T},
 \ \ \ \mbox{ as }\ \epsilon\to 0.
\end{eqnarray*}
Here without this condition, we rigorously prove the instability of $\mathcal{U}_{0}$ under the condition that $P'(0)<0$.
This improves the result in \cite{Yi-21}.

\end{remark}

\section{Application to $n$-patch models with two-dimensional kinetic system}
\label{sec-app}

As an application of Theorem \ref{lm-unstab-patch},
we consider  an $n$-patch model with two-dimensional kinetic system
\begin{equation} \label{2d-patch-model}
\begin{aligned}[cll]
&\frac{d u_{j}}{dt} =\delta\sum_{i\in\Omega}\left(d_{11}(u_{i}-u_{j})+d_{12} (v_{i}-v_{j})\right)+f(u_{j},v_{j},\alpha),\  \ \ \  \
   &&\ \ j\in \Omega:=\{1,2,...,n\},\\
&\frac{d v_{j}}{dt} =\delta \sum_{i\in\Omega}(d_{21}(u_{i}-u_{j})+d_{22} (v_{i}-v_{j}))+g(u_{j},v_{j},\alpha),\  \ \ \  \
  &&\ \  j\in \Omega:=\{1,2,...,n\},
\end{aligned}
\end{equation}
where $(u_{j}(t),v_{j}(t))^{T}\in\mathbb{R}^{2}$ represent the population densities of two species in the $j$-th patch,
$n$ is an integer greater or equal to $2$, and $\delta>0$ indicates the coupling strength.
Here the parameters  $d_{ij}$  in \eqref{2d-patch-model} indicate the diffusion coefficients.
In particular, when  $i\neq j$, the parameters $d_{ij}$ are the cross-diffusion rates.

The underlying kinetic system of the above patch model reads as the following system of ordinary differential
equations
\begin{equation}
\label{2d-ODE}
\begin{aligned}[cl]
&\dot u=f(u,v,\alpha),\\
&\dot v =g(u,v,\alpha),
\end{aligned}
\end{equation}
where $\alpha$ in $\mathbb{R}$ is a system parameter, and the functions $f$ and $g$ are sufficiently smooth.
If the kinetic system \eqref{2d-ODE} has a stable periodic solution $\phi(t)\in\mathbb{R}^{2}$,
then patch model \eqref{2d-patch-model} has a synchronous periodic solution $((\phi(t))^{T},...,(\phi(t))^{T})^{T}\in\mathbb{R}^{2n}$
 that are also stable in the absence of diffusion.
Our aim is to discuss whether this stable synchronous periodic solution  could become unstable  in the presence  of diffusion.
Specially, we focus on  periodic solutions arising from closed orbits and degenerate Hopf bifurcation,
which are called large- and small-amplitude bifurcating periodic solutions, respectively.

\subsection{Application to large-amplitude bifurcating periodic solutions}
\label{sec:4.1}

Consider the kinetic system with $\alpha=\alpha_{0}$ which is in the form
\begin{equation} \label{2d-ODE-2}
\begin{aligned}[cl]
&\dot u=f(u,v,\alpha_{0})=:f(u,v),\\
&\dot v=g(u,v,\alpha_{0})=:g(v,v).
\end{aligned}
\end{equation}
Let $\psi(t)=(u_{0}(t),v_{0}(t))^{T}\in\mathbb{R}^{2}$ be a periodic solution of \eqref{2d-ODE-2}
that satisfies the following hypothesis:
\begin{itemize}
  \item The periodic solution $\psi(t)$ is stable and has minimum period $p>0$.
  The related Floquet multipliers $\gamma$ and $\tilde{\gamma}$ satisfy
  \begin{eqnarray}\label{cond-FM}
\gamma=1,\ \ \ 0<\tilde{\gamma}<1.
\end{eqnarray}
\end{itemize}
In the view of bifurcation theory,
this periodic solution $\psi(t)$ can
bifurcate from a perturbation of a closed orbit in the  kinetic system  \eqref{2d-ODE} with $\alpha$ near $\alpha_{0}$.

It is clear that the linearization of system \eqref{2d-ODE-2} about $\psi(t)$ is in the form
\begin{eqnarray*}
\frac{d}{dt}
\left(
\begin{array}{c}
\tilde{u}(t)\\
\tilde{v}(t)
\end{array}\right)=
\left(
\begin{array}{cc}
f_{u}(\psi(t)) & f_{v}(\psi(t)) \\
g_{u}(\psi(t)) & g_{v}(\psi(t))
\end{array}
\right)
\left(
\begin{array}{c}
\tilde{u}(t)\\
\tilde{v}(t)
\end{array}\right).
\end{eqnarray*}
Then by Lemma 7.3 in \cite[p.120]{Hale-80}, we have
\begin{eqnarray*}
\tilde{\gamma}=\gamma\tilde{\gamma}=\exp\left(\int_{0}^{p}\left(f_{u}(\psi(t))+g_{v}(\psi(t))\right)dt\right)<1.
\end{eqnarray*}
Following the discussion in Section \ref{sec-mainresult},
we consider an auxiliary planar system of the form
\begin{eqnarray}\label{pert-2D-ODE-1}
(I_{2}+\epsilon D)
\left(
\begin{array}{c}
\dot u\\
\dot v
\end{array}
\right)=
\left(
\begin{array}{c}
f(u,v,\alpha) \\
g(u,v,\alpha)
\end{array}
\right),
\end{eqnarray}
where $I_{2}$ is the $2\times 2$ identity matrix,
and the matrix $D$  is in the form
\begin{eqnarray*}
D=
\left(
\begin{array}{cc}
d_{11} & d_{12}\\
d_{21} & d_{22}
\end{array}
\right).
\end{eqnarray*}
Consider the perturbation  system \eqref{pert-2D-ODE-1} with $\alpha=\alpha_{0}$.
It has the expansion with respect to $\epsilon$ as follows:
\begin{eqnarray}\label{pert-expand}
\ \ \
\begin{aligned}
\dot u&=f(u,v)-\epsilon(d_{11}f(u,v)+d_{12}g(u,v))+O(\epsilon^{2})=:f(u,v)+\epsilon f_{1}(u,v)+O(\epsilon^{2}),\\
\dot v&=g(u,v)-\epsilon(d_{21}f(u,v)+d_{22}g(u,v))+O(\epsilon^{2})=:g(u,v)+\epsilon g_{1}(u,v)+O(\epsilon^{2}).
\end{aligned}
\end{eqnarray}
We summarize some results on periodic solutions bifurcating from $\psi(t)$ in system \eqref{pert-expand}.
\begin{lemma}\label{lm-poincare-bif}
Suppose that the kinetic system \eqref{2d-ODE-2}
has a stable periodic solution satisfying the conditions in \eqref{cond-FM}.
Then there exists a small constant $\epsilon_{0}>0$
such that for each $\epsilon$ with $0\leq |\epsilon|<\epsilon_{0}$,
system \eqref{pert-expand} has exactly one limit cycle
$\psi(t,\epsilon):=(u_{0}(t,\epsilon),v_{0}(t,\epsilon))\in\mathbb{R}^{2}$ with period $P(\epsilon)$
bifurcating from $\psi(t)$.
Moreover, the period function $P(\epsilon)$ has the expansion of the form
\begin{eqnarray}\label{Perd-expansion}
\ \ \ \ \ \ \
\begin{aligned}
P(\epsilon) =&\  p+\epsilon\int_{0}^{p}
   \left(\left(\frac{2fg(f_{u}-g_{v})+(g^{2}-f^{2})(f_{u}+g_{v})}{(f^{2}+g^{2})^{3/2}}|_{(u,v)^{T}=\psi(t)}\right) \right.\\
   &\  \times \frac{1}{\sqrt{f^{2}(\psi(t))+g^{2}(\psi(t))}} \left(I(t)+\frac{1}{1-\tilde{\gamma}}I(p)e^{h(t)} \right)\\
  & \left.-\left(\frac{ff_{1}+gg_{1}}{f^{2}+g^{2}}|_{(u,v)^{T}=\psi(t)}\right)\right)dt+O(\epsilon^{2})\\
  =:&\ p+ P_{1}\epsilon+O(\epsilon^{2}),\ \ \ \ \ 0\leq |\epsilon|<\epsilon_{0},
\end{aligned}
\end{eqnarray}
where $f_{1}$ and $g_{1}$ are defined as in \eqref{pert-expand}, and  $I(t)$ and $h(t)$ are in the form
\begin{eqnarray}
I(t) \!\!\!&=&\!\!\! e^{h(t)}\int_{0}^{t}\left(e^{-h(s)}(fg_{1}-gf_{1})|_{(u,v)^{T}=\psi(s)}\right)ds,\label{df-I-t}\\
h(t) \!\!\!&=&\!\!\! \int_{0}^{t}\left(f_{u}(\psi(s))+g_{v}(\psi(s))\right)ds.\nonumber
\end{eqnarray}
\end{lemma}
\begin{proof}
We can prove the existence of perturbed periodic solutions  using \cite[Theorem 2.1, p.352]{Cod-Lev-55}.
Now we give the expansion of the period function $P(\epsilon)$.
For sufficiently small $|\epsilon|$, by the formulas (2.16), (2.18) and (4.5) in \cite{Urabe-54},
we have
\begin{eqnarray*}
P(\epsilon)\!\!\!&=&\!\!\!p+\int_{0}^{p}
   \left(\left(\frac{2fg(f_{u}-g_{v})+(g^{2}-f^{2})(f_{u}+g_{v})}{(f^{2}+g^{2})^{3/2}}|_{(u,v)^{T}=\psi(t)}\right)
 \rho_{1}(t,\epsilon)\right.\\
   \!\!\!& &\!\!\! \left.-\epsilon\left(\frac{ff_{1}+gg_{1}}{f^{2}+g^{2}}|_{(u,v)^{T}=\psi(t)}\right)\right)dt+O(\epsilon^{2}),
\end{eqnarray*}
where $\rho_{1}(t,\epsilon)$ is in the form
\begin{eqnarray*}
\rho_{1}(t,\epsilon)\!\!\!&=&\!\!\!
 \frac{\epsilon}{\sqrt{f^{2}(\psi(t))+g^{2}(\psi(t))}}I(t) +\frac{\sqrt{f^{2}(\psi(0))+g^{2}(\psi(0))}}{\sqrt{f^{2}(\psi(t))+g^{2}(\psi(t))}}e^{h(t)}c(\epsilon),\\
c(\epsilon) \!\!\!&=&\!\!\! \epsilon \frac{1}{(1-\tilde{\gamma})\sqrt{f^{2}(\psi(0))+g^{2}(\psi(0))}}I(p)+O(\epsilon^{2}).
\end{eqnarray*}
Thus, we can compute \eqref{Perd-expansion}. This finishes the proof.
\end{proof}

Now we state the results on destabilization of large-amplitude periodic solutions.
\begin{proposition}\label{thm-Poincare-Destab}
Suppose that the kinetic system \eqref{2d-ODE} with $\alpha=\alpha_{0}$ has a stable periodic solution
$\psi(t)=(u_{0}(t),v_{0}(t))\in\mathbb{R}^{2}$
that satisfies \eqref{cond-FM}.
If the constant $P_{1}$ defined in \eqref{Perd-expansion} satisfies $P_{1}<0$,
then for sufficiently small $\delta>0$,
the corresponding  synchronous periodic solution $(\psi(t),...,\psi(t))^{T}$
is unstable with respect to  patch model \eqref{2d-patch-model}.
\end{proposition}
\begin{proof}
If $P_{1}<0$, then we have
\begin{eqnarray*}
P'(0)<0.
\end{eqnarray*}
Therefore, the proof is finished by Theorem \ref{lm-unstab-patch}.
\end{proof}

We remark that the formula in Lemma \ref{lm-poincare-bif} gives an analytic formula to determine the sign of $P_{1}$,
although the expression of $P_{1}$ is complicated.
This formula also provides a possibility to numerically give criteria
for the destabilization of synchronous periodic solutions.
As an easy by-product of Proposition \ref{thm-Poincare-Destab},
we have that
the destabilization of synchronous periodic solutions
can not be induced by the identical diffusion rates.
More precisely, we have the following result.

\begin{proposition}\label{thm-TStab-patch}
Suppose that the kinetic system \eqref{2d-ODE} with $\alpha=\alpha_{0}$ has a stable periodic solution
$\psi(t)=(u_{0}(t),v_{0}(t))\in\mathbb{R}^{2}$ that satisfies \eqref{cond-FM}.
Then for sufficiently small $\delta>0$,
the corresponding  synchronous periodic solution $(\psi(t),...,\psi(t))^{T}$ is stable with respect to  patch model \eqref{2d-patch-model}
with the identical diffusion rates.
\end{proposition}
\begin{proof}
We first prove that the period function $P(\epsilon)$ satisfies $P'(0)>0$.
If $d_{11}=d_{22}=d_{0}>0$ and $d_{12}=d_{21}=0$,
then the functions $f_{1}$ and $g_{1}$ in \eqref{pert-expand} are in the form
\begin{eqnarray*}
f_{1}(u,v)=-d_{0}f(u,v), \ \ \ g_{1}(u,v)=-d_{0}g(u,v), \ \ \ (u,v)\in\mathbb{R}^{2}.
\end{eqnarray*}
This yields
\begin{eqnarray*}
f(u,v)g_{1}(u,v)-g(u,v)f_{1}(u,v)=0.
\end{eqnarray*}
Then  $I(t)$  in \eqref{df-I-t} satisfies $I(t)\equiv 0$ for $t\in\mathbb{R}$.
So we can compute
\begin{eqnarray}
\begin{aligned}
P(\epsilon) =p+ \epsilon d_{0}p+O(\epsilon^{2}).
\end{aligned}
\end{eqnarray}
This implies $P'(0)>0$.

Now we prove the stability of $(\psi(t),...,\psi(t))^{T}$ with respect to  patch model \eqref{2d-patch-model}.
Let the notations be defined as in Section \ref{sec-mainresult}.
Since $P'(0)>0$, by Lemma \ref{lm-PF-patch} we have that
the Floquet multipliers $\gamma_{j}(\delta)$ satisfy $\gamma_{1}(\delta)\equiv 1$ for $\delta\geq 0$
and  $\gamma'_{j}(0)<0$ for $j=2,3,...,n$.
Consequently, $\gamma_{j}(\delta)<1$ for sufficiently small $\delta>0$ and $j=2,...,n$.
This together with the fact that $\Psi(p,\delta)$ is continuous with respect to $\delta$
and the condition \eqref{cond-FM}  yields  that
all of the Floquet multipliers of \eqref{eq-patch-var} have modulii less than one except $\gamma_{1}(\delta)$.
Then $(\psi(t),...,\psi(t))^{T}$ is stable with respect to  patch model \eqref{2d-patch-model}
with the identical diffusion rates. This finishes the proof.
\end{proof}

\subsection{Application to small-amplitude bifurcating periodic solutions}
\label{sec-small}

In this section,
we study the destabilization of periodic solutions arising from  Hopf bifurcation.
Based on the normal form theory and the  formal series method,
we give the conditions under which the destabilization of Hopf bifurcating periodic solutions appears.

Without loss of generality, throughout this section  we make the following hypotheses:
\begin{itemize}
\item
The functions $f$ and $g$ in the kinetic system \eqref{2d-ODE} are $C^{\infty}$ in $(u,v,\alpha)$,
and $f(0,0,\alpha)=g(0,0,\alpha)=0$ for all $\alpha\in \mathbb{R}$.
\item The kinetic system \eqref{2d-ODE} has a center-type equilibrium at the origin $O:=(0,0)^{T}$ for $\alpha=0$,
that is, the Jacobian matrix $J(O)$ of system \eqref{2d-ODE} with $\alpha=0$ at the origin
has a pair of purely imaginary eigenvalues $\lambda_{1,2}=\pm {\bf i}\mu_{0}$ for $\mu_{0}>0$.
\end{itemize}

Let $J(\alpha)$ denote the Jacobian matrix of the kinetic system \eqref{2d-ODE} at the origin, that is,
\begin{equation*}
J(\alpha)=\left(
\begin{array}{cc}
f_{u}(0,0,\alpha) &  f_{v}(0,0,\alpha)\\
g_{u}(0,0,\alpha) &  g_{v}(0,0,\alpha)
\end{array}
\right)=:
\left(
\begin{array}{cc}
J_{11}(\alpha) &  J_{12}(\alpha)\\
J_{21}(\alpha) &  J_{22}(\alpha)
\end{array}
\right).
\end{equation*}
Then we can compute
\begin{eqnarray}\label{eq-ep=0}
J_{11}(0)+J_{22}(0)=0, \ \  \
J_{11}(0)J_{22}(0)-J_{12}(0)J_{21}(0)=\mu_{0}^{2}>0.
\end{eqnarray}
By \cite[Lemma 1.1, p.384]{Chow-etal-94},
the kinetic system \eqref{2d-ODE} with $\alpha=0$ in complex coordinates
has the following Poincar\'e-Birkhoff normal form
\begin{eqnarray*}
\dot z={\bf i}\mu_{0}z+C_{1}z^{2}\bar{z}+C_{2}z^{3}\bar{z}^{2}+\cdot\cdot\cdot+C_{k}z^{k+1}\bar{z}^{k}+O(|z|^{2k+3}).
\end{eqnarray*}
The constants $C_{j}$ are called the {\it $j$th Lyapunov coefficients} of system \eqref{2d-ODE} with $\alpha=0$
at the center-type equilibrium $O$.
If the  Lyapunov coefficients $C_{j}$ satisfy
\begin{eqnarray}\label{df-mult}
{\rm Re}(C_{1})=\cdot\cdot\cdot={\rm Re}(C_{k-1})=0, \ \ \  {\rm Re}(C_{k})\neq 0,
\end{eqnarray}
then we say that the kinetic system \eqref{2d-ODE} could undergo a  Hopf bifurcation of order $k$ for $k\geq 1$ at the origin,
and the origin is a weak focus of order $k$.

Recall that the auxiliary system is defined by \eqref{pert-2D-ODE-1}.
When $\epsilon=0$,
system \eqref{pert-2D-ODE-1}  is reduced to the kinetic system \eqref{2d-ODE}.
It is clear that for sufficiently small $|\epsilon|$,
one can transform system \eqref{pert-2D-ODE-1} into the following system
\begin{eqnarray}\label{pert-2D-ODE-2}
\left(
\begin{array}{c}
\dot u\\
\dot v
\end{array}
\right)=
(\mathcal{M}(\epsilon))^{-1}
\left(
\begin{array}{cc}
1+\epsilon d_{22} & -\epsilon d_{12} \\
-\epsilon d_{21} & 1+\epsilon d_{11}
\end{array}
\right)
\left(
\begin{array}{c}
f(u,v,\alpha) \\
g(u,v,\alpha)
\end{array}
\right),
\end{eqnarray}
where
$$\mathcal{M}(\epsilon)=\det(I_{2}+\epsilon D)=(d_{11}d_{22}-d_{12}d_{21})\epsilon^{2}+(d_{11}+d_{22})\epsilon+1.$$
Note that system \eqref{pert-2D-ODE-2} always has an equilibrium at the origin
for all $\alpha\in \mathbb{R}$.
A direct computation yields that
the Jacobian matrix $J(\alpha,\epsilon)$ of system \eqref{pert-2D-ODE-2} at the origin  is in the form
\begin{equation*}
J(\alpha,\epsilon)=
(\mathcal{M}(\epsilon))^{-1}
\left(
\begin{array}{cc}
1+\epsilon d_{22} & -\epsilon d_{12} \\
-\epsilon d_{21} & 1+\epsilon d_{11}
\end{array}
\right)
\left(
\begin{array}{cc}
J_{11}(\alpha) &  J_{12}(\alpha)\\
J_{21}(\alpha) &  J_{22}(\alpha)
\end{array}
\right),
\end{equation*}
and the trace $\mathcal{T}(\alpha,\epsilon)$ of $J(\alpha,\epsilon)$
and the determinant $\mathcal{D}(\alpha,\epsilon)$ of $J(\alpha,\epsilon)$  are given by
\begin{eqnarray}\label{df-D-T}
\begin{aligned}
\mathcal{T}(\alpha,\epsilon)=&\
    (\mathcal{M}(\epsilon))^{-1}\left\{J_{11}(\alpha)+J_{22}(\alpha)\right.\\
    &\left.+\epsilon\left(d_{22}J_{11}(\alpha)+d_{11}J_{22}(\alpha)-d_{12}J_{21}(\alpha)-d_{21}J_{12}(\alpha)\right)\right\},\\
\mathcal{D}(\alpha,\epsilon)=&\
    (\mathcal{M}(\epsilon))^{-1}\det(J(\alpha))=(\mathcal{M}(\epsilon))^{-1}(J_{11}(\alpha)J_{22}(\alpha)-J_{12}(\alpha)J_{21}(\alpha)).
\end{aligned}
\end{eqnarray}
Next we state the results on the periodic solutions bifurcating from the origin.
\begin{lemma}\label{lm-Hopf-bif}
Let $\lambda_{1,2}(\alpha)=A(\alpha)\pm{\bf i}B(\alpha)$ denote the eigenvalues of the matrix $J(\alpha)$.
Suppose that the kinetic system \eqref{2d-ODE} with $\alpha=0$
has a weak focus of order $k$ at the origin, and $A'(0)\neq 0$.
Then the following statements hold:
\begin{enumerate}
\item[{\bf (i)}]
There exist two small constants $\epsilon_{0}>0$ and $\tilde{r}_{0}>0$,
and a smooth function $\alpha(r_{0},\epsilon)$ for $0<r_{0}<\tilde{r}_{0}$ and $|\epsilon|\leq \epsilon_{0}$
such that system \eqref{pert-2D-ODE-2} with $\alpha=\alpha(r_{0},\epsilon)$
has exactly one limit cycle $(u(t,r_{0},\epsilon),v(t,r_{0},\epsilon))$ with period $P(r_{0},\alpha(r_{0},\epsilon))$ near the origin.

\item[{\bf (ii)}]
If the $k$th Lyapunov coefficient $C_{k}$ satisfies  ${\rm Re}(C_{k})<0$ (resp. ${\rm Re}(C_{k})>0$),
then the perturbed limit cycle is stable (resp. unstable),
and each perturbed limit cycle $(u(t,r_{0},\epsilon),v(t,r_{0},\epsilon))$  passes through $(r_{0},0)$.
\end{enumerate}
\end{lemma}
\begin{proof}
Let $\lambda_{1,2}(\alpha,\epsilon)=A(\alpha,\epsilon)\pm{\bf i}B(\alpha,\epsilon)$
denote the eigenvalues of the Jacobian matrix  $J(\alpha,\epsilon)$.
System \eqref{pert-2D-ODE-2} can be
transformed into
\begin{eqnarray} \label{2D-1}
\begin{aligned}
\dot u  &= A(\alpha,\epsilon)u-B(\alpha,\epsilon)v+f_{1}(u,v,\alpha,\epsilon),
\\
\dot v  &= B(\alpha,\epsilon)u+A(\alpha,\epsilon)v+g_{1}(u,v,\alpha,\epsilon).
\end{aligned}
\end{eqnarray}
See the detailed proof in \cite[Appendix A]{Yi-21}.
Set $(u,v)=(r\cos\theta,r\sin\theta)$. Then
\begin{eqnarray} \label{2D-2}
\begin{aligned}
\dot r  &= A(\alpha,\epsilon)r
      +\cos\theta\cdot f_{1}(r\cos\theta,r\sin\theta,\epsilon)+\sin\theta\cdot g_{1}(r\cos\theta,r\sin\theta,\epsilon),
\\
r\dot \theta  &= B(\alpha,\epsilon) r
     +\cos\theta\cdot g_{1}(r\cos\theta,r\sin\theta,\epsilon)-\sin\theta\cdot f_{1}(r\cos\theta,r\sin\theta,\epsilon).
\end{aligned}
\end{eqnarray}
Let $(r(t,r_{0},\alpha,\epsilon),\theta(t,r_{0},\alpha,\epsilon))$
denote the solution of \eqref{2D-2}  with $(r(0),\theta(0))=(r_{0},0)$.
Then for sufficiently small $|\alpha|+|\epsilon|$,
we can define the displacement map by
\begin{eqnarray*}
H(r_{0},\alpha,\epsilon)=r(2\pi,r_{0},\alpha,\epsilon)-r_{0}.
\end{eqnarray*}
Since $H(0,\alpha,\epsilon)=0$ for all $\alpha$ and $\epsilon$,
we write  $H(r_{0},\alpha,\epsilon)$ as
\begin{eqnarray*}
H(r_{0},\alpha,\epsilon)=r_{0}\tilde{H}(r_{0},\alpha,\epsilon),
\end{eqnarray*}
where $\tilde{H}(r_{0},\alpha,\epsilon)$ has the expansion as the form
\begin{eqnarray*}
\tilde{H}(r_{0},\alpha,\epsilon)=\left(\exp\left(2\pi\frac{A(\alpha,\epsilon)}{B(\alpha,\epsilon)}\right)-1\right)
      +H_{1}(\alpha,\epsilon)r_{0}+O(r_{0}^{2}).
\end{eqnarray*}
By  \eqref{df-mult} and \eqref{2D-2}, we further have
\begin{eqnarray}\label{formu-1}
\tilde{H}(r_{0},0,0)=\frac{2\pi {\rm Re}(C_{k})}{\mu_{0}}r_{0}^{2k}+O(r_{0}^{2k+1}).
\end{eqnarray}
A direct computation yields
\begin{eqnarray}\label{formu-2}
\begin{aligned}
\frac{\partial \tilde{H}}{\partial \alpha}(0,0,0)
=&\ \frac{\partial \tilde{H}}{\partial \alpha}(0,\alpha,0)|_{\alpha=0}\\
=&\ \frac{\partial }{\partial \alpha}\left(\exp\left(2\pi\frac{A(\alpha,0)}{B(\alpha,0)}\right)-1\right)|_{\alpha=0}\\
=&\  \frac{2\pi A'(0)}{\mu_{0}}\neq 0.
\end{aligned}
\end{eqnarray}
This together with   the implicit function theorem yields that
there exist two small constants $\epsilon_{0}>0$ and $\tilde{r}_{0}>0$,
and a smooth function $\alpha(r_{0},\epsilon)$ for $0<r_{0}<\tilde{r}_{0}$ and $|\epsilon|\leq \epsilon_{0}$
such that
$
\tilde{H}(r_{0},\alpha(r_{0},\epsilon),\epsilon)=0
$
for sufficiently small $|r_{0}|+|\epsilon|$.
Thus, we obtain {\bf (i)}.
The statements in {\bf (ii)} can be proved by \eqref{2D-2}.
Therefore, the proof is now complete.
\end{proof}

By the above lemma, we can further verify that there exists a constant $\alpha_{0}>0$
such that for each $(\alpha,\epsilon)$ with $|\alpha|<\alpha_{0}$ and $|\epsilon|<\epsilon_{0}$,
system \eqref{pert-2D-ODE-2} has a small-amplitude periodic orbit bifurcating from the origin.
Let $T(\alpha,\epsilon)$ denote the corresponding period.
Then we call $T(\alpha,\epsilon)$ the {\it period function} for this family of periodic solutions
arising from Hopf bifurcation.

Note that the period function $T(\alpha,\epsilon)$ depends not only on $\alpha$ but also on $\epsilon$.
Then the results in \cite{Chen-Huang-22,Gasull-05,Hassard-81}, where the period function depends on a single parameter,
is not applicable to  $T(\alpha,\epsilon)$.
Now we give the formula of $T(\alpha,\epsilon)$ in the next lemma.

\begin{lemma} \label{lm-Perd-Hopf}
Let $T(\alpha,\epsilon)$ denote the period function of the Hopf bifurcating periodic solutions.
Then $T(\alpha,\epsilon)$ satisfies the formula
\begin{eqnarray}\label{T-expansion}
\frac{2\pi}{T(\alpha,\epsilon)}=B(\alpha,\epsilon)+\sum_{j=1}^{k} {\rm Im}(C_{j}(\alpha,\epsilon))r_{0}^{2j}+O(r_{0}^{2k+1}),
\end{eqnarray}
where $r_{0}$ satisfies
\begin{eqnarray}\label{r0-expansion}
\begin{aligned}
r_{0}^{2k}=R_{0}(\alpha,\epsilon)=&
-\frac{\epsilon}{2{\rm Re}(C_{k})}\left(d_{22}J_{11}(0)+d_{11}J_{22}(0)-d_{12}J_{21}(0)-d_{21}J_{12}(0)\right)\\
& -\frac{A'(0)}{{\rm Re}(C_{k})}\alpha
+(|(\alpha,\epsilon)|^{2}).
\end{aligned}
\end{eqnarray}
for sufficiently small $|\alpha|+|\epsilon|$.
\end{lemma}
\begin{proof}
By \cite[Lemma 1.1, p.384]{Chow-etal-94},
system \eqref{2D-1} can be normalized into
\begin{eqnarray*}
\dot z=(A(\alpha,\epsilon)+{\bf i}B(\alpha,\epsilon))z+\sum_{j=1}^{k}C_{j}(\alpha,\epsilon)z^{j+1}\bar{z}^{j}+O(|z|^{2k+3}).
\end{eqnarray*}
Let $z(t,r_{0},\epsilon)$ denote the bifurcating periodic solution with $z(0,r_{0},\epsilon)=r_{0}$.
Define
\[
\tau=t/T(\alpha,\epsilon),\ \ \ \ \ z(t,r_{0},\epsilon)=r_{0}e^{2\pi {\bf i}\tau}\eta(\tau,r_{0},\epsilon),
\]
where $\eta(\tau+1,r_{0},\epsilon)=\eta(\tau,r_{0},\epsilon)$ for each $\tau\in\mathbb{R}$.
Then
\begin{eqnarray}\label{eq-eta}
(2\pi{\bf i})\eta+\frac{d\eta}{d\tau}
=T(\alpha,\epsilon)\eta \left\{\lambda(\alpha,\epsilon)+\sum_{j=1}^{k}C_{j}(\alpha,\epsilon)r_{0}^{2j}(\eta\bar{\eta})^{j}\right\}+O(r_{0}^{2k+1}).
\end{eqnarray}
We expand $\eta(\tau,r_{0},\epsilon)$ as the form
\begin{eqnarray*}
\eta(\tau,r_{0},\epsilon)=\eta_{0}(\tau)+\eta_{1}(\tau,r_{0},\epsilon)+\eta_{2}(\tau,r_{0},\epsilon)
     +\cdot\cdot\cdot+\eta_{k}(\tau,r_{0},\epsilon)+\cdot\cdot\cdot,
\end{eqnarray*}
where $\eta_{j}(\tau,\alpha,\epsilon)$  are periodic functions with period one,
the homogeneous polynomials of $j$-th degree with respect to $r_{0}$ and $\epsilon$,
and satisfy
\[
\eta_{0}(0)=1,\ \ \ \ \eta_{j}(0,r_{0},\epsilon)=0,\ \ \ j\geq 1.
\]
Substituting the expansion of $\eta(\tau,r_{0},\epsilon)$ into \eqref{eq-eta} and then comparing the term of the zeroth degree with respect to $r_{0}$ and $\epsilon$,
we have
\[
(2\pi{\bf i})\eta_{0}+\frac{d\eta_{0}}{d\tau}=(2\pi{\bf i})\eta_{0}, \ \ \ \ \eta_{0}(0)=1.
\]
This yields $\eta_{0}(\tau)\equiv 1$ for $\tau\in\mathbb{R}$.
Comparing the term of the first degree yields
\[
(2\pi{\bf i})\eta_{1}+\frac{d\eta_{1}}{d\tau}=(2\pi{\bf i})\eta_{1}+d_{1}(\alpha,\epsilon), \ \ \ \ \eta_{1}(0)=0,
\]
where $d_{1}(\alpha,\epsilon)$ is a constant term.
Recall that $\eta_{1}(\tau,r_{0},\epsilon)$ is periodic with respect to $\tau$.
Then $\eta_{1}(\tau,r_{0},\epsilon)=0$ for each $\tau\in \mathbb{R}$.
Similarly, we can prove that $\eta_{j}(\tau,r_{0},\epsilon)=0$ for $2\leq j\leq k$.
This together with \eqref{eq-eta} yields \eqref{T-expansion}.
By applying the implicit function theorem, \eqref{formu-1}, \eqref{formu-2} and
\begin{eqnarray*}
\frac{\partial \tilde{H}}{\partial \epsilon}(0,0,0)
\!\!\!&=&\!\!\! \frac{\partial \tilde{H}}{\partial \epsilon}(0,0,\epsilon)|_{\epsilon=0}\\
\!\!\!&=&\!\!\!\frac{\partial }{\partial \epsilon}\left(\exp\left(2\pi\frac{A(0,\epsilon)}{B(0,\epsilon)}\right)-1\right)|_{\epsilon=0}\\
\!\!\!&=&\!\!\!
\frac{\pi}{\mu_{0}}\left(d_{22}J_{11}(0)+d_{11}J_{22}(0)-d_{12}J_{21}(0)-d_{21}J_{12}(0)\right),
\end{eqnarray*}
we obtain \eqref{r0-expansion}. This finishes the proof.
\end{proof}

Finally, we establish the existence and destabilization of Hopf bifurcating periodic solutions
for patch model \eqref{2d-patch-model}.

\begin{proposition}\label{thm-Hopf-exist}
Suppose that
the kinetic system \eqref{2d-ODE} with $\alpha=0$ has a weak focus of order $k$ at the origin, $A'(0)\neq 0$,
and the related Lyapunov coefficients $C_{j}$ satisfy
\begin{eqnarray*}
{\rm Re}(C_{1})=\cdot\cdot\cdot={\rm Re}(C_{k-1})=0, \ \  {\rm Re}(C_{k})<0.
\end{eqnarray*}
Then there exists a constant $\alpha_{0}>0$ such that
\begin{enumerate}
\item[{\bf (i)}] if $|\alpha|<\alpha_{0}$ and $\alpha A'(0)>0$,
then the kinetic system \eqref{2d-ODE} has a stable periodic solution
$\psi(t,\alpha):=(u(t,\alpha),v(t,\alpha))$ bifurcating from the origin,
and $\psi(t,\alpha)$ tends to the origin as $\alpha\to 0$.

\item[{\bf (ii)}] if $|\alpha|<\alpha_{0}$ and $\alpha A'(0)>0$,
then patch model \eqref{2d-patch-model} has
a  synchronous periodic solution $\mathcal{U}(t,\alpha):=(\psi(t,\alpha),...,\psi(t,\alpha))^{T}\in\mathbb{R}^{2n}$.
\end{enumerate}
\end{proposition}
\begin{proof}
Let $\epsilon=0$ in system \eqref{pert-2D-ODE-2}.
Since $A'(0)\neq 0$ and  the $k$th Lyapunov coefficient $C_{k}$ satisfies  ${\rm Re}(C_{k})<0$,
by the formulas \eqref{formu-1} and \eqref{formu-2} we have
\begin{eqnarray} \label{alp-exp}
\alpha(r_{0},0)=-\frac{{\rm Re}(C_{k})}{A'(0)}r_{0}^{2k}+O(r_{0}^{2k+1}).
\end{eqnarray}
By the proof for  Lemma \ref{lm-Hopf-bif},
there exists a sufficiently small $\alpha_{0}>0$
such that the kinetic system \eqref{2d-ODE} with $|\alpha|<\alpha_{0}$ and $\alpha A'(0)>0$
has a stable periodic solution $(u(t,\alpha),v(t,\alpha))$ bifurcating from the origin.
This implies the existence  of perturbed periodic solutions for \eqref{2d-patch-model}.
Therefore,  the proof is now complete.
\end{proof}

\begin{proposition}\label{thm-Hopf-Destab}
Suppose that the conditions in Theorem \ref{thm-Hopf-exist} hold.
If $|\alpha|<\alpha_{0}$, $\alpha A'(0)>0$,
\begin{eqnarray}\label{cond-ImCk}
{\rm Im}(C_{1})=\cdot\cdot\cdot={\rm Im}(C_{k_{1}-1})=0,& \ \ \  {\rm Im}(C_{k_{1}})\neq 0,
\end{eqnarray}
where $k_{1}$ satisfies  $1\leq k_{1}\leq k$,
and one of the following two conditions holds:
\begin{enumerate}
\item[{\bf (C1)}]
 $k_{1}<k$ and
\[
{\rm Im}(C_{k_{1}})
\left(d_{22}J_{11}(0)+d_{11}J_{22}(0)-d_{12}J_{21}(0)-d_{21}J_{12}(0)\right)>0.
\]

\item[{\bf (C2)}] $k_{1}=k$ and
\[
\mu_{0}(d_{11}+d_{22})+\frac{{\rm Im}(C_{k}(0,0))}{{\rm Re}(C_{k}(0,0))}
\left(d_{22}J_{11}(0)+d_{11}J_{22}(0)-d_{12}J_{21}(0)-d_{21}J_{12}(0)\right)<0.
\]
\end{enumerate}
then there exists a small constant  $\hat{\alpha}_{0}$ with $0<\hat{\alpha}_{0}<\alpha_{0}$
such that for each $\alpha$ with $0<|\alpha|<\hat{\alpha}_{0}$ and  sufficiently small $\delta>0$,
the synchronous periodic solution $\mathcal{U}(t,\alpha)$ is unstable with respect to  patch model \eqref{2d-patch-model}.
\end{proposition}
\begin{proof}
We first compute $\frac{\partial B}{\partial \epsilon}(0,0)$.
It is clear that $$B(\alpha,\epsilon)=\frac{\sqrt{4\mathcal{D}(\alpha,\epsilon)-(\mathcal{T}(\alpha,\epsilon))^{2}}}{2},$$
where $\mathcal{D}(\alpha,\epsilon)$ and $\mathcal{T}(\alpha,\epsilon)$ are defined by \eqref{df-D-T}.
Then
\begin{eqnarray*}
\frac{\partial B}{\partial \epsilon}(0,0)
\!\!\!&=&\!\!\!
\frac{4\frac{\partial \mathcal{D}}{\partial \epsilon}(\alpha,\epsilon)-2\mathcal{T}(\alpha,\epsilon)\frac{\partial \mathcal{T}}{\partial \epsilon}(\alpha,\epsilon)}
{4\sqrt{4\mathcal{D}(\alpha,\epsilon)-(\mathcal{T}(\alpha,\epsilon))^{2}}}|_{(\alpha,\epsilon)=(0,0)}\\
\!\!\!&=&\!\!\!
-\frac{\mu_{0}}{2}(d_{11}+d_{22}).
\end{eqnarray*}
Now we compute $\frac{\partial T}{\partial \epsilon}(0,0)$.
When $1\leq k_{1}<k$, by Lemma \ref{lm-Perd-Hopf} and \eqref{cond-ImCk} we have
\begin{eqnarray*}
&&-\frac{2\pi}{T^{2}(\alpha,0)}\frac{\partial T}{\partial \epsilon}(\alpha,0)\\
&&=
\frac{\partial B}{\partial \epsilon}(\alpha,0)
+\frac{k_{1}}{k}{\rm Im}(C_{k_{1}}(0,0))(R_{0}(\alpha,0))^{\frac{k_{1}}{k}-1}\frac{\partial R_{0}}{\partial \epsilon}(\alpha,0)+O(|\alpha|)\\
&&=
\left\{-\frac{k_{1}{\rm Im}(C_{k_{1}}(0,0))}{2k{\rm Re}(C_{k}(0,0))}
\left(d_{22}J_{11}(0)+d_{11}J_{22}(0)-d_{12}J_{21}(0)-d_{21}J_{12}(0)\right)+O(\alpha)\right\}\\
&&\ \ \ \times (R_{0}(\alpha,0))^{\frac{k_{1}}{k}-1}
 -\frac{\mu_{0}}{2}(d_{11}+d_{22})+O(\alpha),
\end{eqnarray*}
where $R_{0}(\alpha,0)$ is defined as in \eqref{r0-expansion}.
Note that $R_{0}(\alpha,0)\to 0^{+}$ as $\alpha \to 0$ and $k_{1}<k$.
Then
under the conditions that ${\rm Re}(C_{k})<0$ and ${\bf (C1)}$,
we have that $\frac{\partial T}{\partial \epsilon}(\alpha,0)<0$ for sufficiently small $|\alpha|$.

When $k_{1}=k$, by Lemma \ref{lm-Perd-Hopf} and \eqref{cond-ImCk} we have
\begin{eqnarray*}
&&\lefteqn{-\frac{2\pi}{T^{2}(0,0)}\frac{\partial T}{\partial \epsilon}(0,0)}\\
&&=
\frac{\partial B}{\partial \epsilon}(0,0)+{\rm Im}(C_{k}(0,0))\frac{\partial R_{0}}{\partial \epsilon}(0,0)\\
&&=
-\frac{\mu_{0}}{2}(d_{11}+d_{22})-\frac{{\rm Im}(C_{k}(0,0))}{2{\rm Re}(C_{k}(0,0))}
\left(d_{22}J_{11}(0)+d_{11}J_{22}(0)-d_{12}J_{21}(0)-d_{21}J_{12}(0)\right).
\end{eqnarray*}
Then under the condition {\bf (C2)},
we have that $\frac{\partial T}{\partial \epsilon}(\alpha,0)<0$ for sufficiently small $|\alpha|$.
Therefore, the proof is finished by Theorem \ref{lm-unstab-patch}.
\end{proof}

\begin{remark}
Yi \cite{Yi-21} recently applied the results in  \cite{Hassard-81} to give the period function
for small-amplitude periodic solutions bifurcating from a weak focus of order one,
whose  first Lyapunov coefficient has nonzero real part.
Following that, Yi \cite{Yi-21} gave a criterion for the destabilization of
the synchronous periodic solutions
arising from a weak focus of order one.
Here we consider the destabilization of the synchronous periodic solutions
arising from a higher-order weak focus.
This phenomenon is called degenerate Hopf bifurcation  \cite{Far-etal-89}.
It is worth mentioning that \cite{Chen-Huang-22,Gasull-05,Hassard-81} considered
the period function of perturbed periodic solutions
appearing in one-parameter Hopf bifurcation.
However, there are two parameters involved in determining the related period function for a higher-order weak focus,
the results in \cite{Chen-Huang-22,Gasull-05,Hassard-81} are not applicable to this case.
\end{remark}

\subsection{Application to the two-patch Holling-Tanner model}
\label{sec:4.3}

In this section,
we consider the two-patch Holling-Tanner model as illustration for our results.
A similar argument can be also applied to explore the destabilization of synchronous periodic solutions
for various patch models with two-dimensional kinetic systems.
It is worth mentioning that Theorem \ref{lm-unstab-patch} are applicable to
patch models not only with two-dimensional kinetic systems  but also with high-dimensional kinetic systems,
e.g. epidemic models \cite{Lu-Gao-Huang-Wang-23,Wang-Mulone-03},
ecological systems \cite{Arumugam-2021,Jiang-Lam-Lou-21},
chemical reaction models \cite{Dolnik-Epstein-93,Moore-Horsthemke-05}, etc.

Consider a two-dimensional kinetic system
\begin{eqnarray} \label{PP-model}
\begin{aligned}
\frac{d u}{d t} &= \dot u = ru\left(1-\frac{u}{K}\right)-q(u)v=:P(u,v),
\\
\frac{d v}{d t} &=  \dot v = sv\left(1-h\frac{v}{u}\right)=:Q(u,v),
\end{aligned}
\end{eqnarray}
where $u$ and $v$ are the population densities of a prey and a predator, respectively.
Here $r$ and $s$  indicate the growth rates of the prey and the predator respectively,
$K$ measures the prey environmental carrying capacity in the absence of predation,
and $h$ presents a measure of food quality.
The functional response $q(u)$ is of Holling type II  and has the form
\[
q(u)=\frac{mu}{u+a},
\]
where we require $a>0$ and $m>0$.
This is the classical Holling-Tanner model,
which exhibits interesting oscillatory behaviors (see, for instance, \cite{Gasull-Kooij-Torregrosa-97,Hsu-Huang-95}).

To study spatial aspects,
we consider a two-patch model with the kinetic system \eqref{PP-model} on each patch
and cross-diffusion-like couplings between the two patches.
The dynamics are governed by
\begin{equation} \label{2d-pp-patch-model}
\begin{aligned}[cll]
&\frac{d u_{j}}{dt} =\delta\sum_{i\in\Omega}\left(d_{11}(u_{i}-u_{j})+d_{12} (v_{i}-v_{j})\right)+P(u_{j},v_{j}),\  \ \ \  \
   &&\ \ \ \ j\in \Omega:=\{1,2\},\\
&\frac{d v_{j}}{dt} =\delta \sum_{i\in\Omega}(d_{21}(u_{i}-u_{j})+d_{22} (v_{i}-v_{j}))+Q(u_{j},v_{j}),\  \ \ \  \
   &&\ \ \ \ j\in \Omega:=\{1,2\},
\end{aligned}
\end{equation}
where $d_{ij}$  are the diffusion coefficients and $\delta$ is the coupling strength.
Then the linearization of patch model \eqref{2d-pp-patch-model} about a synchronous periodic solution
$\mathcal{U}_{0}(t)=(\phi(t),\phi(t))$ is
\begin{equation} \label{eq:example-linear}
\begin{aligned}[cll]
&\frac{d X_{j}}{dt} =\delta\sum_{i\in\Omega}\left(d_{11}(X_{i}-X_{j})+d_{12} (Y_{i}-Y_{j})\right)+P_{u}(\phi(t))X_{j}+P_{v}(\phi(t))Y_{j},
   &&\ \  j\in \Omega:=\{1,2\},\\
&\frac{d Y_{j}}{dt} =\delta \sum_{i\in\Omega}(d_{21}(X_{i}-X_{j})+d_{22} (Y_{i}-Y_{j}))+Q_{u}(\phi(t))X_{j}+Q_{v}(\phi(t))Y_{j},
   &&\ \  j\in \Omega:=\{1,2\},
\end{aligned}
\end{equation}

In the following,
we use two concrete examples to  demonstrate the destabilization of synchronous periodic solutions for patch model \eqref{2d-pp-patch-model}.
Examples 1 and 2 illustrate the cases of large- and small-amplitude bifurcating periodic solutions, respectively.
Here we shall use Lyapunov exponents (see \cite{Guck-Holmes-83}) of the linearized systems
to describe the destabilization. If the linearized systems have a positive Lyapunov exponent,
then the corresponding synchronous periodic solutions become unstable.

\vskip 0.3cm

{\bf Example 1.}
Consider the Holling-Tanner model  \eqref{PP-model} and  fix the parameters as follows:
\begin{eqnarray*}
a=1,\ \ \ h=0.5,\ \ \ K=5,\ \ \ m=1,\ \ \ r=1,\ \ \ s=0.1.
\end{eqnarray*}
Numerical simulation   with MATLAB shows that the kinetic system \eqref{PP-model} has a stable periodic solution $\varphi_{1}(t)$ for $t\in\mathbb{R}$
that surrounds a unstable focus. See Figure \ref{fg-periodic-smu}.
\begin{figure}[!h]
  \centering
  \includegraphics[width=3in,height=2.4in]{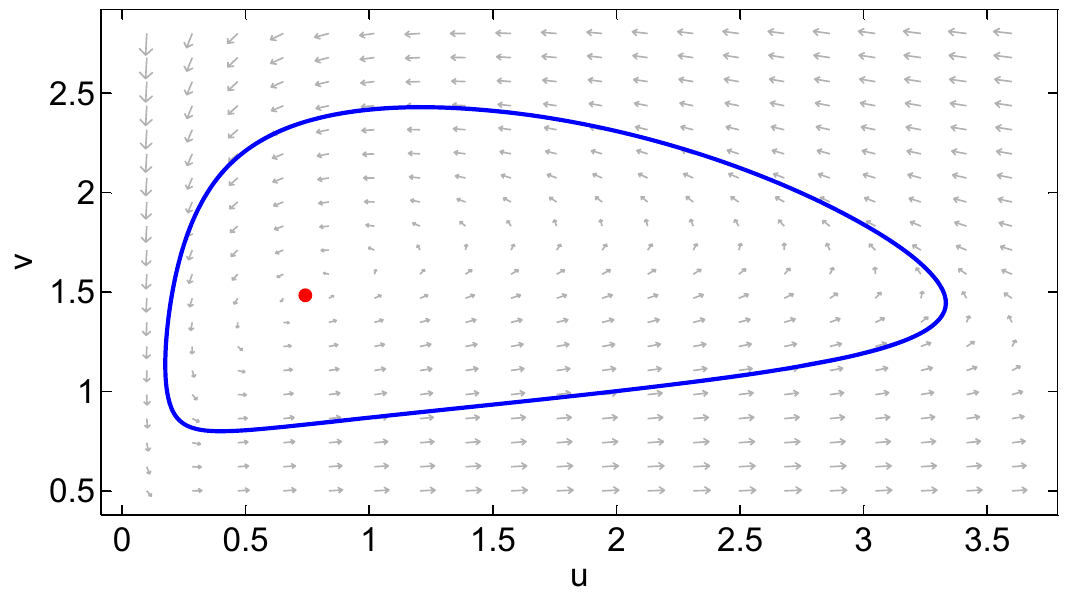}
  \caption{
  A stable periodic solution (blue cycle) surrounds a unstable focus (red dot) in model \eqref{PP-model},
  where $a=1$, $h=0.5$,  $K=5$, $m=1$, $r=1$ and $s=0.1$.
  }
  \label{fg-periodic-smu}
\end{figure}

Set $\mathcal{U}_{1}(t):=(\varphi_{1}(t),\varphi_{1}(t))$.
When $\delta=0$,
patch model \eqref{2d-pp-patch-model} is decoupled and the synchronous periodic solution $\mathcal{U}_{1}$ is stable
with respect to patch model \eqref{2d-pp-patch-model}.
 With the aid of MATLAB,
we obtain the results as follows:
\begin{enumerate}
\item[(i)] Set $\delta=0.01$, $d_{11}=d_{22}=1$ and $d_{12}=d_{21}=0$.
Numerical simulation shows that the largest Lyapunov exponent is zero. See Figure \ref{fig-1-2}.
This coincides with Proposition \ref{thm-TStab-patch},
which shows that the synchronous periodic solution $\mathcal{U}_{1}$ is still stable with respect
to patch model \eqref{2d-pp-patch-model} with the identical diffusion rates.

\item[(ii)] Set $\delta=0.1$,  $d_{11}=d_{21}=d_{22}=1$ and $d_{12}=10$.
Numerical simulation shows that the largest Lyapunov exponent is about  0.0031. See Figure \ref{fig-1-3}.
This implies the destabilization of the synchronous periodic solution $\mathcal{U}_{1}$ can be induced by cross-diffusion-like couplings.
\end{enumerate}

\begin{figure}[!h]
\centering
\subfigure[]{
\begin{minipage}[t]{1\linewidth}
\centering
\includegraphics[width=3.5in,height=2.6in]{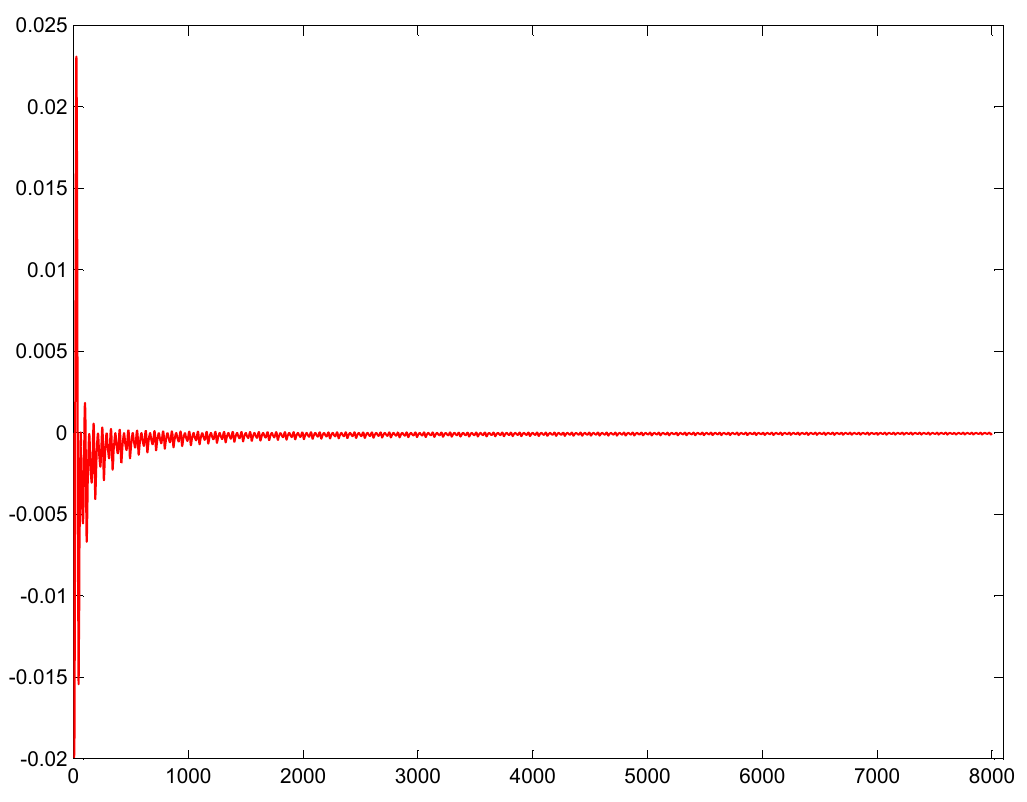}
\label{fig-1-2}
\end{minipage}
}%
\\
\subfigure[]{
\begin{minipage}[t]{1\linewidth}
\centering
\includegraphics[width=3.5in,height=2.6in]{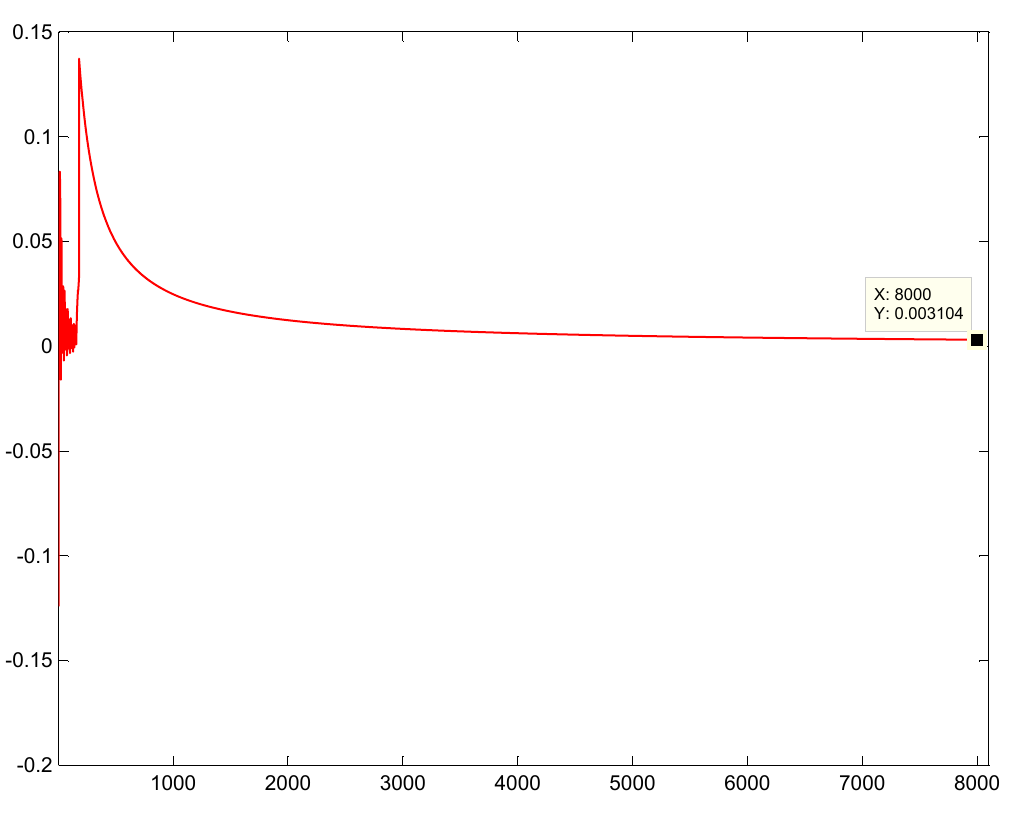}
\label{fig-1-3}
\end{minipage}
}%
\centering
\caption{
The largest Lyapunov exponents of the linearized system \eqref{eq:example-linear}.
Fix $a=1$, $h=0.5$,  $K=5$, $m=1$, $r=1$ and $s=0.1$.
(a)  The largest Lyapunov exponent is zero when $\delta=0.01$, $d_{11}=d_{22}=1$ and $d_{12}=d_{21}=0$.
(b) The largest Lyapunov exponent is about 0.0031 when $\delta=0.01$,  $d_{11}=d_{21}=d_{22}=1$ and $d_{12}=10$.
}
\label{Fg-no-cross}
\end{figure}

{\bf Example 2.}
Consider the Holling-Tanner model  \eqref{PP-model} and fix $K=m=r=1$.
By a direct computation,
the positive equilibria of  model  \eqref{PP-model} are determined by the roots of equation
\begin{eqnarray}\label{eq:exmp-2-1}
\beta u^2 + (a\beta -\beta + s)u - a\beta=0, \ \ \ \beta:=hs.
\end{eqnarray}
Clearly, the above equation has a unique positive root that is denoted by $u_{*}$.
So $E_{*}:=(u_{*},u_{*}/h)$ is the unique positive equilibrium of  model  \eqref{PP-model}.
Furthermore, if the following equations hold:
\begin{eqnarray}
2u_{*}^{2}+(a+s-1)u_{*}+as\!\!\!&=&\!\!\! 0,\label{eq:exmp-2-2}\\
-(3+a)u_{*}^{3}-6au_{*}(a+u_{*})+a^{2}(1-a)\!\!\!&=&\!\!\! 0, \label{eq:exmp-2-3}
\end{eqnarray}
then by \cite[Theorem 3.2]{Gasull-Kooij-Torregrosa-97} and statements in \cite[p.161]{Gasull-Kooij-Torregrosa-97},
this equilibrium $(u_{*},u_{*}/h)$ is a weak focus of multiplicity two and asymptotically stable.
Consequently, the corresponding first- and second-order Lyapunov coefficients $C_{1}$ and $C_{2}$ satisfy
\[
{\rm Re}(C_{1})=0,\ \ \ {\rm Re}(C_{2})<0.
\]

Set $s=0.1$ and let $u$ in \eqref{eq:exmp-2-1} be replaced by $u_{*}$.
Solving  \eqref{eq:exmp-2-1}, \eqref{eq:exmp-2-2} and \eqref{eq:exmp-2-3} yields a unique positive solution
$a\approx0.336238$, $h\approx 0.222132$ and $u_{*}\approx 0.085693$.
Now we fix these $a$ and $h$, i.e., $a\approx0.336238$ and $h\approx 0.222132$, and vary $s$.
At this equilibrium $E_{*}$,
we can obtain the Jacobian matrix $J(s)=(J_{ij}(s))$ of model \eqref{PP-model} satisfies
\begin{eqnarray*}
J_{11}(0.1)\approx0.1,\ \ \ J_{12}(0.1)\approx-0.203097, \ \ \ J_{21}(0.1)\approx 0.450183, \ \ \ J_{22}(0.1)\approx -0.1,
\end{eqnarray*}
and has a pair of purely imaginary eigenvalue $\lambda_{\pm}\approx \pm  0.285361 \sqrt{-1}$.
By the formula of the first-order Lyapunov coefficient $C_{1}$ in page 90 of \cite{Hassard-81}
(see also Appendix A in \cite{Yi-21}),
we can compute that the imaginary part ${\rm Im}(C_{1})$ of the first-order Lyapunov coefficient $C_{1}$ is
\begin{eqnarray*}
{\rm Im}(C_{1})\approx -1.872272.
\end{eqnarray*}
Note that the trace $\mathcal{T}(s)$ of $J(s)$ is $(0.1-s)$.
Then by Proposition \ref{thm-Hopf-exist},
there exists a sufficiently small $s_{0}>0$ such that for $0<0.1-s<s_{0}$,
model \eqref{PP-model} has a stable periodic solution $\varphi_{2}(t)$ arising from $E_{*}$.
See Figure \ref{fg-Hopf-smu}.
\begin{figure}[!htbp]
  \centering
  \includegraphics[width=3in,height=2.4in]{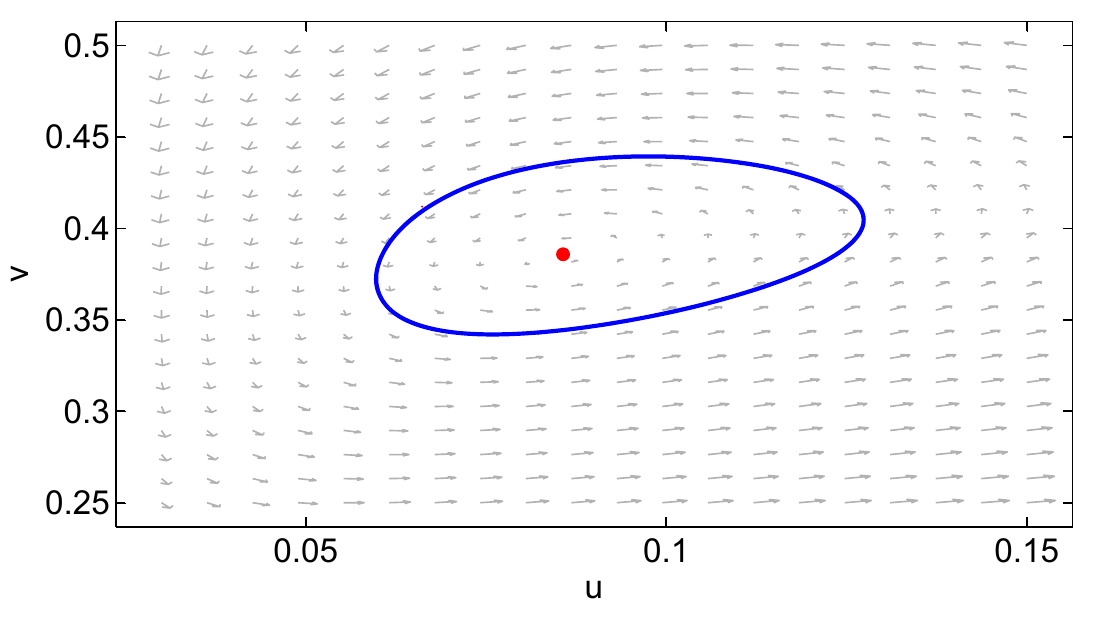}
  \caption{A stable periodic solution (blue cycle) bifurcates from a weak focus $E_{*}$ (red dot) in model \eqref{PP-model},
   where $a=0.3362380612$, $h=0.2221316654$,  $K=1$, $m=1$, $r=1$ and $s=0.09999$.
  }
  \label{fg-Hopf-smu}
\end{figure}

Set $\mathcal{U}_{2}(t):=(\varphi_{2}(t),\varphi_{2}(t))$.
When $\delta=0$,
patch model \eqref{2d-pp-patch-model} is decoupled and the synchronous periodic solution $\mathcal{U}_{2}$ is stable
with respect to patch model \eqref{2d-pp-patch-model}.
If the diffusion rates $d_{ij}$ satisfy
\begin{eqnarray}\label{examp-cond}
d_{22}J_{11}(0.1)+d_{11}J_{22}(0.1)-d_{12}J_{21}(0.1)-d_{21}J_{12}(0.1)<0,
\end{eqnarray}
then by Proposition \ref{thm-Hopf-Destab},
the synchronous periodic solution $\mathcal{U}_{2}$ becomes unstable.
 With the aid of MATLAB,
we obtain the results as follows:
\begin{enumerate}
\item[(i)] Set $\delta=0.01$,  $d_{11}=1$, $d_{12}=1$,  $d_{21}=-100$ and  $d_{22}=5$.
Numerical simulation shows that the largest Lyapunov exponent is zero. See Figure \ref{fig-2-1}.
Then the synchronous periodic solution $\mathcal{U}_{2}$ is still stable.

\item[(ii)] Set $\delta=0.01$,  $d_{11}=1$, $d_{12}=1$,  $d_{21}=100$ and  $d_{22}=5$.
Numerical simulation shows that the largest Lyapunov exponent is about   0.0079. See Figure \ref{fig-2-2}.
Then the synchronous periodic solution $\mathcal{U}_{2}$ becomes unstable.
This coincides with Proposition \ref{thm-Hopf-Destab},
which shows that stable synchronous periodic solution $\mathcal{U}_{1}$ becomes unstable with respect
to patch model \eqref{2d-pp-patch-model} if the condition \eqref{examp-cond} holds.
\end{enumerate}

\begin{figure}[!htbp]
\centering
\subfigure[]{
\begin{minipage}[t]{1\linewidth}
\centering
\includegraphics[width=3.5in,height=2.6in]{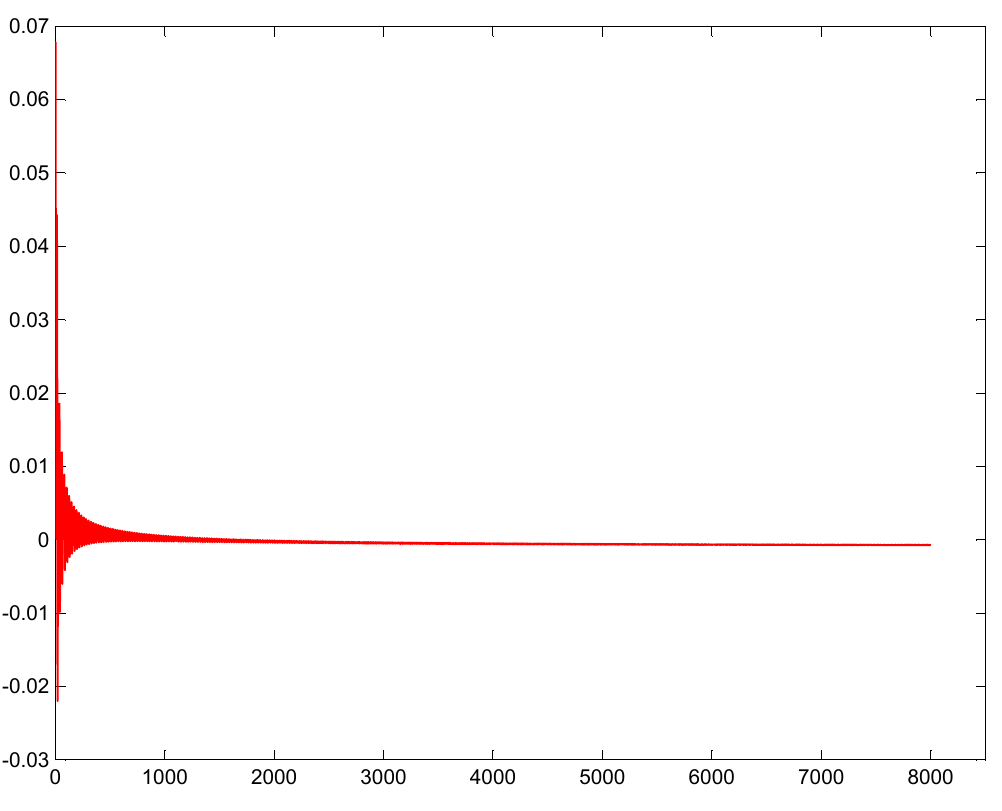}
\label{fig-2-1}
\end{minipage}
}%
\\
\subfigure[]{
\begin{minipage}[t]{1\linewidth}
\centering
\includegraphics[width=3.5in,height=2.6in]{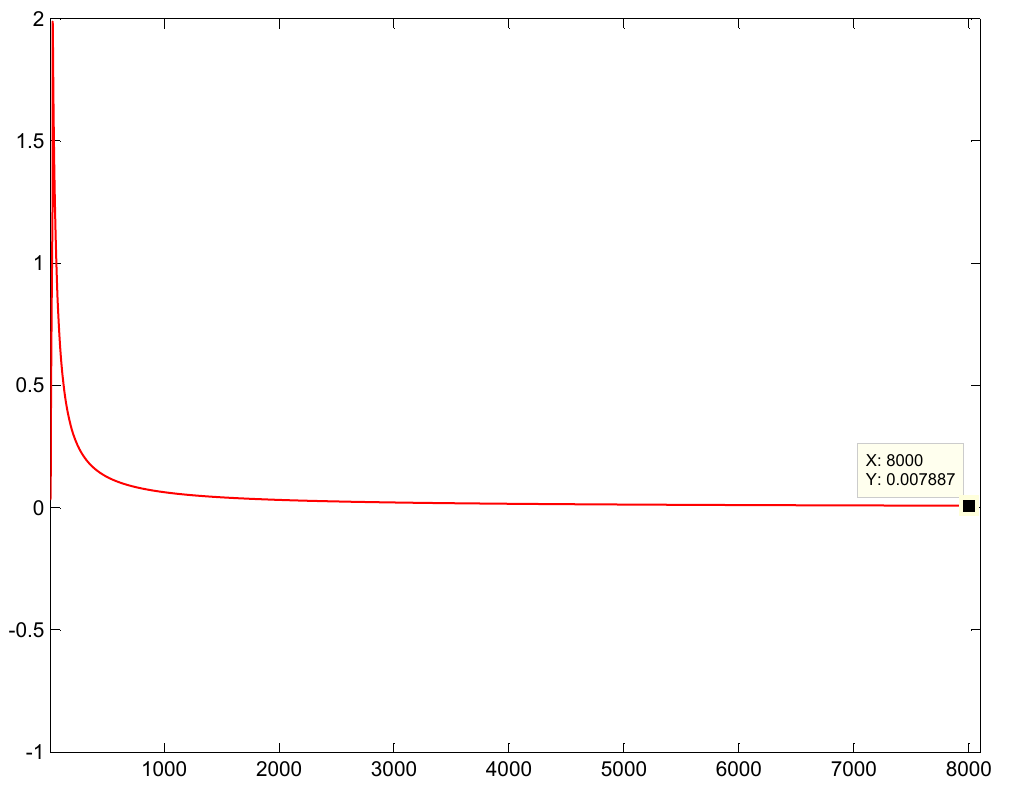}
\label{fig-2-2}
\end{minipage}
}%
\centering
\caption{
The largest Lyapunov exponents of the linearized system \eqref{eq:example-linear}.
Fix $a=0.3362380612$, $h=0.2221316654$,  $K=1$, $m=1$, $r=1$, $s=0.09999$,
$\delta=0.01$,  $d_{11}=1$, $d_{12}=1$ and  $d_{22}=5$.
(a)  The largest Lyapunov exponent for the linearized system \eqref{eq:example-linear}  is zero
when $d_{21}=-100$.
(b) The largest Lyapunov exponent for the linearized system \eqref{eq:example-linear}  is about  0.0079,
when $d_{21}=100$.
}
\label{Fg-Hopf-cross}
\end{figure}

\section*{Appendix A. perturbation of eigenvalues for matrices}
\label{sec:pert-eigenvalues}

In this appendix, we present the perturbation theory for matrices developed in our recent work \cite{Chen-Huang-23}.
For each $\zeta=(\zeta_{1},...,\zeta_{N})^{T}$ in $\mathbb{C}^{N}$ or $\mathbb{R}^{N}$,
set
\begin{eqnarray*}
|\zeta|:=\max\{|\zeta_{1}|,|\zeta_{2}|,...,|\zeta_{N}|\}.
\end{eqnarray*}
 Let $\|\cdot\|$ denote the norm of a matrix, i.e., the maximum row sum of the absolute values of the entries.

In order to give  our criterion for  the destabilization of synchronous periodic solutions,
it is necessary to give the asymptotic expressions of the Floquet spectra.
The argument is based on two perturbation results which were proved in our recent work \cite{Chen-Huang-23}.

Consider a matrix $A \in\mathbb{R}^{N\times N}$ ($N\geq 2$) as follows:
\begin{eqnarray*}
A=\left(
\begin{array}{cc}
A_{1} & 0\\
0 & A_{2}
\end{array}
\right),
\end{eqnarray*}
where $A_{1}={\rm diag}(a_{1},a_{2},...,a_{N_{1}})\in\mathbb{R}^{N_{1}\times N_{1}}$
and $A_{2}\in\mathbb{R}^{N_{2}\times N_{2}}$ for integers $N_{1}$ and $N_{2}$ with $0<N_{1}=N-N_{2}<N$.
Additionally, the spectra $\sigma(A_{1})$ and $\sigma(A_{2})$ of $A_{1}$ and $A_{2}$
satisfy the following assumption:
\begin{itemize}
  \item
The spectra $\sigma(A_{1})$ and $\sigma(A_{2})$  are separated by
a simple closed positively oriented cycle $\Gamma$ in the complex plane,
and $\sigma(A_{1})$ lies in the interior of the closed cycle $\Gamma$.
\end{itemize}

For a sufficiently small constant $\delta_{0}>0$,
let $B:(-\delta_{0},\delta_{0}) \to \mathbb{R}^{N\times N}$ denote a matrix function
that is analytic in $\delta\in (-\delta_{0},\delta_{0})$
and satisfies that $\|B(\delta)\|=O(|\delta|)$ for sufficiently small $|\delta|$.
Consider a perturbation of $A$ as follows:
\renewcommand\theequation{A.1}
\begin{eqnarray}\label{df-A-delta-0}
A(\delta):=A+B(\delta),\ \ \ \ \ \delta\in(-\delta_{0},\delta_{0}).
\end{eqnarray}
By the classical perturbation theory of eigenvalues for matrices (see, for instance, \cite[pp. 63-64]{Kato-80}),
the eigenvalues of $A(\delta)$ are continuous in $\delta$ .
By continuity,
we can choose sufficiently small $\delta_{0}>0$ such that
all eigenvalues of $A(\delta)$ perturbed from $\sigma(A_{1})$ lie in the interior of the closed cycle $\Gamma$
and all others lie outside the domain surrounded by $\Gamma$ in the complex plane.
Following that, we can define a family of parametric projections $\mathcal{P}(\delta)$ by
\[
\mathcal{P}(\delta)=\frac{1}{2\pi{\bf i}} \oint_{\Gamma}(\lambda I_{N}-A(\delta))\, d\lambda,
\ \ \ \ \ \delta\in (-\delta_{0},\delta_{0}).
\]
Concerning these parametric projections $\mathcal{P}(\delta)$, we have the next lemma.

\setcounter{theorem}{0}
\renewcommand{\thetheorem}{A.\arabic{theorem}}
\begin{lemma}\label{lm-app-0} \cite[Lemma A.1]{Chen-Huang-23}
There exists a  sufficiently small constant $\hat{\delta}_{0}$ with $0<\hat{\delta}_{0}\leq \delta_{0}$
such that the following statements hold:
\begin{enumerate}
\item[{\bf (i)}]
$\mathcal{P}(\delta)$ is analytic in the interval $(-\hat{\delta}_{0},\hat{\delta}_{0})$.

\item[{\bf (ii)}]
Let the ranges of $\mathcal{P}(0)$ and $I-\mathcal{P}(0)$ be spanned by
$\{\xi_{j}: j=1,...,N_{1}\}$ and $\{\xi_{j}: j=N_{1}+1,...,N\}$, respectively.
Then there exists an operator-valued function $U(\cdot):(-\hat{\delta}_{0},\hat{\delta}_{0})\to\mathbb{R}^{N\times N}$
which is analytic and invertible,
such that for each $\delta\in (-\hat{\delta}_{0},\hat{\delta}_{0})$, the sets
$\{U(\delta)\xi_{j}: j=1,...,N_{1}\}$ and $\{U(\delta)\xi_{j}: j=N_{1}+1,...,N\}$
form the bases of the ranges of $\mathcal{P}(\delta)$ and $I-\mathcal{P}(\delta)$, respectively.
\end{enumerate}
\end{lemma}

We continue to study a special case of the perturbation \eqref{df-A-delta-0}, i.e.,
 the perturbation $A(\delta)$ is in the form
\renewcommand\theequation{A.2}
\begin{eqnarray}\label{df-A-delta}
A(\delta)=\left(
\begin{array}{cc}
A_{1}(\delta) & 0\\
0 & A_{2}(\delta)
\end{array}
\right)+A_{3}(\delta),
\ \ \ \ \ \delta\in(-\delta_{0},\delta_{0}),
\end{eqnarray}
where $A_{1}(\delta)\in\mathbb{R}^{N_{1}\times N_{1}}$,
$A_{2}(\delta)\in\mathbb{R}^{N_{2}\times N_{2}}$ and $A_{3}(\delta)\in\mathbb{R}^{N\times N}$
are continuous and satisfy
\begin{eqnarray*}
A_{1}(\delta)=A_{1}+O(|\delta|),
\ \ \
A_{2}(\delta)=A_{2}+O(|\delta|), \ \ A_{3}(\delta)=O(|\delta|^{2})
\end{eqnarray*}
for sufficiently small $|\delta|$.
Then we have the following result.

\renewcommand{\thetheorem}{A.\arabic{theorem}}
\begin{lemma}\cite[Lemma A.2]{Chen-Huang-23} \label{lm-app}
Consider the perturbation $A(\delta)$ of the form \eqref{df-A-delta}.
Suppose that
\begin{eqnarray*}
A_{1}(\delta)=\left(
\begin{array}{cccc}
a_{1}+b_{1}\delta & 0 & \cdot\cdot\cdot & 0\\
0 & a_{2}+b_{2}\delta  & \cdot\cdot\cdot & 0\\
\cdot\cdot\cdot & \cdot\cdot\cdot & \cdot\cdot\cdot & \cdot\cdot\cdot \\
0 & 0 & \cdot\cdot\cdot & a_{N_{1}}+b_{N_{1}}\delta
\end{array}
\right)
,\ \ \ \delta\in(-\delta_{0},\delta_{0}),
\end{eqnarray*}
where  $b_{j}$ ($j=1,2,...,N_{1}$) are real constants.
Then for each $j=1,2,...,N_{1}$, the eigenvalue $\lambda_{j}(\delta)$ perturbed from $\lambda_{j}=a_{j}$
has the asymptotic expression as follows:
\begin{eqnarray*}
\lambda_{j}(\delta)=a_{j}+b_{j}\delta+O(\delta^{2})
\end{eqnarray*}
for sufficiently small $|\delta|\geq 0$.
\end{lemma}


\section*{\bf{Declarations}}

{\bf{Ethical Approval}}

 Not applicable.

 \

{\bf{Competing interests}}

 There are no financial or non-financial competing interests.

\

{\bf{Authors' contributions}}

S. Chen and J. Huang wrote the paper. All authors read and approved the manuscript.

\

{\bf{Funding}}

This work was partly supported by the National Natural Science Foundation of China (Grant No. 12101253, 12231008)
and the Scientific Research Foundation of CCNU (Grant No. 31101222044).

\

{\bf{Availability of data and materials}}

Data sharing not applicable to this article as no datasets were
generated or analysed during the current study.

\bibliographystyle{amsplain}

\end{document}